\documentclass[reqno]{amsart}
\usepackage[scale=0.75, centering, headheight=14pt]{geometry}
\usepackage[latin1]{inputenc}
\usepackage[T1]{fontenc}
\usepackage{lmodern}
\usepackage[english]{babel}
\usepackage{esint}
\usepackage{mathtools}

\DeclarePairedDelimiter\floor{\lfloor}{\rfloor}

\usepackage{tikz}
\usepackage{bbm}
\usepackage{amsmath,amssymb,amsfonts,amsthm}
\usepackage{mathtools,accents}
\usepackage{mathrsfs}
\usepackage{xfrac}
\usepackage{array} 
\usepackage{aliascnt}
\usepackage{booktabs} % for much better looking tables
\usepackage{array} % for better arrays (eg matrices) in maths

\usepackage{verbatim} % adds environment for commenting out blocks of text & for better verbatim
\usepackage{subfig} % make it possible to include more than one captioned figure/table in a single float

\usepackage{mathrsfs, dsfont}
\usepackage{amssymb}
\usepackage{amsthm}
\usepackage{amsmath,amsfonts,amssymb,esint}
\usepackage{graphics,color}
\usepackage{enumerate}
\usepackage{mathtools,centernot}

\usepackage{microtype}
\usepackage{paralist} % very flexible & customisable lists (eg. enumerate/itemize, etc.)
\usepackage{cases}
\usepackage[initials]{amsrefs}
\allowdisplaybreaks

\usepackage{braket}
\usepackage{bm}

\usepackage[citecolor=blue,colorlinks]{hyperref}
\addto\extrasenglish{}

\usepackage{enumerate}
\usepackage{xcolor}
\usepackage{aliascnt}

\makeatletter
\def\newaliasedtheorem#1[#2]#3{
  \newaliascnt{#1@alt}{#2}
  \newtheorem{#1}[#1@alt]{#3}
  \expandafter\newcommand\csname #1@altname\endcsname{#3}
}
\makeatother

\theoremstyle{plain}
\newtheorem{theorem}{Theorem}[section]
\newaliasedtheorem{lemma}[theorem]{Lemma}
\newaliasedtheorem{prop}[theorem]{Proposition}
\newaliasedtheorem{claim}[theorem]{Claim}
\newaliasedtheorem{corollary}[theorem]{Corollary}
\theoremstyle{remark}

\theoremstyle{definition}
\newaliasedtheorem{definition}[theorem]{Definition}
\newaliasedtheorem{example}[theorem]{Example}

\theoremstyle{remark}
\newaliasedtheorem{remark}[theorem]{Remark}

\numberwithin{equation}{section}

\def\eps{\varepsilon}

\def\R{\mathbb R}
\def\N{{\mathbb N}}% nonnegative integers
\def\Z{{\mathbb Z}}% integers
\def\T{{\mathbb T}}% torus
\def\Q{{\mathbb Q}}% rational numbers

\DeclareMathOperator{\diver}{div}

\DeclareMathOperator{\supp}{supp}

\newcommand*{\RR}{\ensuremath{\mathcal{R}}}

\newcommand{\nop}{12}
\DeclareMathOperator{\AC}{AC([0,1];\T^d)}
\newcommand{\no}{10}

\newcommand{\lam}{\lambda_q^{d(1+\alpha)+2}}
 \newcommand{\XX}{{\mbox{\boldmath$X$}}}

\DeclareMathOperator{\dist}{d}
\title{Almost everywhere non-uniqueness of integral curves for divergence-free Sobolev vector fields}

\author[ J. Pitcho, M. Sorella]{J. Pitcho, M. Sorella}
\address{Jules Pitcho
\hfill\break  Universit\"at Z\"urich, Instit\"ut f\"ur Mathematik, CH-8057 Z\"urich, Switzerland}
\email{jules.pitcho@uzh.ch}
\address{Massimo Sorella
\hfill\break  \'Ecole Polytechnique F\'ed\'erale de Lausanne, Institute of Mathematics, Station 8, CH-1015 Lausanne, Switzerland.}
\email{massimo.sorella@epfl.ch}

\begin{document}

\maketitle

\begin{abstract}
We construct divergence-free Sobolev vector fields in $ C([0,1]; W^{1,r}(\T^d; \R^d))$ with $r<d$ and $d \geq 2$ which simultaneously admit any finite number of distinct positive solutions to the continuity equation. We then show that the vector fields we produce have at least as many integral curves starting from  $\mathscr{L}^d$-a.e. point of $\T^d$ as the number of distinct positive solutions to the continuity equation these vector fields admit. 
%This shows that selection principle for the regular Lagrangian flow of Sobolev vector fields by a compressibility condition is crucial: it selects a single integral curve for the regular Lagrangian flow $\mathscr{L}^d$-a.e. amongst any finite number of integral curves starting from $\mathscr{L}^d$-a.e. 
Our work uses convex integration techniques introduced in  \cite{ML18,BDLC20}  to study non-uniqueness for positive solutions of the continuity equation. We then infer non-uniqueness for integral curves from Ambrosio's superposition principle.
\end{abstract}
%\begin{abstract}
%In this work we construct divergence-free Sobolev vector fields in $ C([0,1]; W^{1,r}(\T^d; \R^d))$ with $r<d$ and $d \geq 2$ which admit any fixed finite number of distinct positive solutions to the continuity equation. This result relies on the convex integration techniques introduced for the continuity equation in \cite{2019AnnPDE} and adapted in \cite{BDLC20} to get positive solutions for the same problem. Finally, using Ambrosio's superposition principle \cite{A04} we infer that the vector fields constructed  have at least as many integral curves starting from  $\mathscr{L}^d$-a.e. point of $\T^d$ as the number of distinct positive solutions to the continuity equation these vector fields admit.  
%\end{abstract}
\par
\medskip\noindent
\textbf{Keywords:}  Sobolev vector fields, generalized flows,  continuity equation, ODE, integral curves. 
\par
\medskip\noindent
{\sc MSC (2020):  35A02 - 35D30 - 35Q49 -34A12.
\par
}

\section{Introduction}\label{sec:intro}

\begin{comment}
This paper is concerned with the non-uniqueness of integral curves for Sobolev vector fields.  

To avoid technicalities we will work on the $d$ dimensional torus, namely we will consider periodic vector fields $u : [0,1] \times \T^d \rightarrow \R^d$. In the sequel we use  $ \mathscr{L}^d$ for the Lebesgue measure on $\T^d.$

\begin{definition}\label{defn:int-curve}
	Let  $u: (0,1) \times \T^d \to \R^d$ be a Borel map. We say that $\gamma \in AC([0,1]; \T^d)$ is an integral curve of $u$ starting at $x$ if $\gamma(0)=x$ and $\gamma'(t) = u(t, \gamma(t))$ for a.e. $t\in [0,T]$.
\end{definition}

This problem arises naturally from \cite[Theorem 1.3]{BDLC20}, where the authors are able to construct a Sobolev vector field such that uniqueness fails for a non vanishing Lebesgue set $A \subset \T^d$. 
\end{comment}

In this paper we
study  positive solutions of the continuity equation
\begin{equation}
\label{eq_continuity}
\begin{cases}
\partial_t \rho  + \diver(\rho u)  = 0,
\\
\rho(\cdot, t)= \rho_0(\cdot)
\end{cases}
\end{equation} 
where $u : [0,1] \times \T^d \rightarrow \R^d$ is a prescribed vector field on the $d$-dimensional torus and $\rho_0 : \T^d \to \R$ is the initial datum. Throughout this work, \eqref{eq_continuity} will be understood in the sense of distributions which only requires that $\rho$ and $\rho u$ be integrable. We then study integral curves of the vector field $u$. 

In the smooth setting, the Cauchy-Lipschitz theory guarantees the existence of a unique flow $\XX:[0,1]\times \T^d\to \T^d$ of the vector field $u$ satisfying
\begin{equation}
\label{eq_flow}
\left\{
\begin{split}
\partial_t \XX(t,x)  & = u(t,\XX(t,x)), \\
\XX(0,x)& = x.
\end{split}
\right.
\end{equation} 
The classical Liouville theorem then gives a representation of solutions of \eqref{eq_continuity} in terms of the flow  $\XX$ of the vector field $u$ through the formula
\begin{equation}\label{eq_rep_formula}
\rho(t, \cdot ) \mathscr{L}^d=\XX(t,.)_\# ( \rho_0\mathscr{L}^d). 
\end{equation}

For rough vector fields, the relationship between the continuity equation and the corresponding flow is an active field of research since the foundational work of DiPerna and Lions in \cite{DPL89}. 
%We briefly recall some previous works in order to introduce our result.
%In recent decades, extensive research activity has been devoted to the relationship between the continuity equation and the corresponding flow for rough vector fields. We give here a brief  summary of some aspects of this research that are needed to explain our result.
By means of a regularization scheme, they showed that if $u \in L^1((0,1);W^{1,r}(\T^d))$ and $\diver u \in L^1((0,1)\times \T^d)$, then \eqref{eq_continuity} is well-posed in the class 
 $L^\infty ((0,1); L^p(\T^d))$, where $ p,r\geq 1 $ satisfy the  relation
$$ \frac{1}{p} + \frac{1}{r} \leq 1.$$
 In \cite{AmbBV}, Ambrosio extended the work of DiPerna and Lions to the setting of BV vector fields.
% Ambrosio later proved in  \cite{AmbBV} that if $u \in L^1((0,1);BV(\T^d))$ and $\diver u (t, \cdot) \ll \mathscr{L}^d$ for  $\mathscr{L}^1$-a.e. $t\in(0,1)$, then \eqref{eq_continuity} is well-posed in $L^\infty((0,1) \times \T^d)$. 
 %\textcolor{red}{The assumption on $\diver (u)$ are different from Ambrosio and DiPerna--Lions.. to be checked, seem weird. Anyway I think we can erase all the part of uniqueness of Ambrosio in BV, because it is not inherent for our purpose.}
 
%Having recalled the relation between the ODE \eqref{eq_flow} and the PDE \eqref{eq_continuity} in the smooth setting
 We now gather some useful definitions to investigate  the relation of the ODE \eqref{eq_flow} and the PDE \eqref{eq_continuity} in the non-smooth setting. 
%For a general Borel vector field, the following is taken as definition for {\em integral curves}. 
\begin{definition}\label{defn:int-curve}
	Let  $u: (0,1) \times \T^d \to \R^d$ be a Borel map. We say that $\gamma \in AC([0,1]; \T^d)$ is an {\em integral curve}  of $u$ starting at $x$ if $\gamma(0)=x$ and $\gamma'(t) = u(t, \gamma(t))$ for a.e. $t\in (0,1)$.
\end{definition}

The {\em regular Lagrangian flow} is then a suitable selection of integral curves of $u$ by a compressibility condition (introduced in \cite{DPL89, AmbBV}). 
 
\begin{definition}[Regular Lagrangian flow]\label{def:reg:flow}
Let $u:(0,1)\times\T^d\to \R^d$ be Borel. We say that a Borel map $\XX:[0,1]\times \R^d\to \R^d$ is a {\em regular Lagrangian flow} of $u$ if 
	\begin{enumerate}[(i)]
		\item for $\mathscr{L}^d$-a.e. $x\in \T^d$, $t\mapsto \XX(t,x)$ is integral curve of $u$ with $\XX(0,x)=x$,
		\item there is a constant $C>0$ such that for every $t\in[0,1]$, $\XX(t,.)_\#\mathscr{L}^d\leq C\mathscr{L}^d$.
	\end{enumerate}
%	$C$ is the compressibility constant of the regular Lagrangian flow $\XX$. 
\end{definition}
%Well-posedness of the regular Lagrangian flow may be understood a weaker, measure-theoretic counterpart to the well-posedness \eqref{eq_continuity}. 
The well-posedness of the regular Lagrangian flow for vector fields $u$ in $L^1((0,1);W^{1,1}(\T^d))$ with the negative part of the divergence satisfying $[\diver u ]^-\in L^1((0,1);L^\infty(\T^d))$ was first derived from the well-posedness of \eqref{eq_continuity} for bounded densities, and for such densities the formula \eqref{eq_flow} holds using as $\XX$ the  regular Lagrangian flow (see \cite{DPL89} and see \cite{AmbBV} for the BV vector fields case). Later in \cite{CripDeL06:estimates}, Crippa and De Lellis proved well-posedness of the regular Lagrangian flow without resorting to the PDE \eqref{eq_continuity}, but their approach only works for vector fields in 
$L^1((0,1);W^{1,r}(\T^d))$ with $r>1$. 
%\red{Formula \eqref{eq_rep_formula} then gives a Lagrangian representation of the unique solution in the DiPerna-Lions or Ambrosio regime in terms of the initial datum and the regular Lagrangian flow. Historically, the regular Lagrangian flow was first constructed precisely as the unique flow satisfying \eqref{eq_rep_formula}.}
% 
At any rate, the uniqueness of the  regular Lagrangian flow does not imply
% Such uniqueness and stability are not related to the ``
$\mathscr{L}^d$- a.e. uniqueness of {\em integral curves}. Indeed, Bru\`e, Colombo and De Lellis recently produced divergence-free Sobolev vector fields -- uniqueness of the  regular Lagrangian flow associated to these vector fields therefore holds -- for which almost everywhere uniqueness of integral curves fails (see \cite[Theorem 1.3]{BDLC20}).
 %indeed the known uniqueness results (for the regular Lagrangian flow) does not imply ``for a.e. x there is a unique integral curve of $u$ starting at $x$''.
% This gave a negative answer to a long-standing open question raised in \cite[p. 546]{DPL89} \cite[p. 231]{AmbBV}, \cite[Section 4]{A08}: is the uniqueness of the regular Lagrangian flow a consequence of the $\mathscr{L}^d$-a.e. uniqueness of integral curves? %  solved recently by \cite[Theorem 1.3]{BDLC20}, even if
 We note that the case of continuous vector field is still open (question posed in  \cite[Section 2.3]{Alberti12}) although in \cite{CC16}, Crippa and Caravenna  proved almost everywhere uniqueness of the trajectories when $ u \in C((0,1); W^{1,r})$ when $r>d$.
 %, under the additional assumption that ``forward-backward integral curves of $u$ are trivial'' (see \cite{CC16}).
%\red{Do you think this Navier-Stokes comment is necessary?}
%(also the case of vector fields that are suitable weak solutions of the Navier-Stokes system had a positive answer in this direction, see \cite{RSNavierAlmost}, \cite{RSNavier}).

In this work  we show that, for divergence-free Sobolev vector fields, uniqueness of integral curves of the ODE \eqref{eq_flow} 
%\red{I think it would be better to talk of integral curves here rather than trajectories} 
can fail for a set of initial data with full measure. In fact, we show that the non-uniqueness for integral curves is even worse: for any natural number $N$ we produce divergence-free Sobolev vector fields with at least $N$ integral curves starting almost everywhere. We highlight that our result demonstrates the power of the selection principle of the  regular Lagrangian flow ((ii) in Definition \ref{def:reg:flow}) for integral curves of Sobolev vector fields. Indeed, amongst at least $N$ integral curves starting from $\mathscr{L}^d$-a.e. point of $\T^d$, the  regular Lagrangian flow selects a \emph{single} integral curve for $\mathscr{L}^d$-a.e. starting point. 

\begin{theorem} \label{t_main}
For every $d,N \in \N$, $ d \geq 2$, $r \in [1,d)$ and $s< \infty$ there is a divergence-free vector field $u \in C((0,1); W^{1,r}(\T^d; \R^d) \cap L^s)$ such that the following holds for every Borel map $v$ with $u = v $ $\mathscr{L}^{d+1}$-a.e.:

(NU) 
For $\mathscr{L}^d$-a.e. $x \in \T^d$ there
 are at least $N$ integral curves of $v$ starting at $x$.
\end{theorem}

 Ambrosio's superposition principle \cite[Theorem 3.2]{A08} bridges the gap between positive solutions of the continuity equation \eqref{eq_continuity} for a vector field $u$ and the integral curves of $u$ (as in Definition \ref{defn:int-curve}): it gives a way of representing positive solutions of the continuity equation in terms of integral curves of the vector field without any differentiability assumption, i.e. under more general assumptions than  DiPerna-Lions theory. 
Using Ambrosio's superposition principle, we will derive Theorem \ref{t_main} from a non-uniqueness result for positive solutions of \eqref{eq_continuity}, which in turn will be proved using a convex integration iterative procedure. The term convex integration is generic to designate iterative techniques by which  {\em wild solutions} of PDEs are constructed. Such techniques  were introduced  in the study of the continuity equation in the groundbreaking work of Modena and 
Sz\'{e}kelyhidi  \cite{ML18,MoSzRenormalized} 
% improved later by Modena and Sattig in \cite{MS20} 
 (see also \cite{BDLLOnsager,BV18,DLEuler,DLL09,DLL13,Is18,CDRS21,CL21,CL20NS,CL20,SG21} for interesting results using the convex integration methods).
%the references for the convex integration method in other PDE problems).

%Our convex integration scheme follows the one introduced in \cite{BDLC20}, but  our scheme
%We do not use the regular Lagrangian flow to conclude the non-uniqueness of integral curves, but we refine the convex integration scheme of \cite{BDLC20} to get a stronger non-uniqueness result for \eqref{eq_continuity}. 
 % products will be a \emph{single} vector field $u$ and a finite number of positive densities $\rho_i$ with mutually disjoint compact supports for times near to $1$ such that $(\rho_i, u)$ weakly solves \eqref{eq_continuity}. This property of the supports of the densities will be important to deduce Theorem \ref{t_main}.

\begin{theorem}\label{p_comp_supp}
%Let $N\in\N$ and let $\{A_i\}_{1 \leq i \leq N}$ be a family of pairwise disjoint supports.
% nonnegative, compactly supported functions on $\T^d$ with pairwise disjoint supports and such that $\int_{\T^d} \phi_i (x) dx=1$ for any $i =1,..,N$. 
Let $d,N \in \N$ with $ d \geq 2$. Let $p \in (1, \infty)$, $r \in [1, \infty]$ be such that
$$\frac{1}{p} + \frac{1}{r} > 1 + \frac{1}{d},$$
and denote by $p'$ the dual exponent of $p$, i.e. $1/{p} + {1}/{p'}=1$.
Then there exists a divergence-free vector field $u\in C^0([0,1], W^{1,r}(\T^d)\cap L^{p'}(\T^d))$ and a family of nonnegative densities $\{\rho_i\}_{1\leq i \leq N} \subset C^0([0,1], L^p(\T^d))$ such that  the following holds

\begin{itemize}
	\item[(i)] the couple $(u,\rho_i)$ weakly solves \eqref{eq_continuity},
	\item[(ii)] for each time $t\in [0,1/3] $, $ \rho_i(t, \cdot )\equiv 1$ for any $i$,
	\item[(iii)] for each time $t\in [2/3,1]$, $\supp\left (\rho_i(t, \cdot ) \right) \cap$ $\supp\left  (\rho_j(t, \cdot ) \right) $ is negligible for any $i \neq j$. Furthermore ${\supp}(\rho_i (t, \cdot ))$ has non-empty interior for any $t \in [0,1]$ and any $i.$
	% for any  $i,j =1,..,N$ $i \neq j$.
\end{itemize}
%where we used the periodized functions $\phi_{\lambda_0,i} (\cdot ) := \phi_i (\lambda_0 \cdot )$ for every $i =1,..,N$.
%\red{the range of r is not consistent between the two theorems}.
\end{theorem} 

%We will use the convex integration scheme for the continuity equation in order to prove the above theorems. 
%The convex integration scheme is an hard iterative procedure to construct {\em wild solutions} for  some PDEs.
%The term building block will refer to a smooth function or vector field which is fixed before iterating the scheme, and which is used to construct perturbations. 
In order to prove Theorem \ref{p_comp_supp}, we adapt the convex integration scheme for positive solutions of the continuity equation introduced in \cite{BDLC20}: our proof makes use of two new ideas. 
Firstly, we keep track of a fixed number of densities and one {\em single} vector field throughout the iteration scheme. Each density is perturbed using a distinct family of building blocks\footnote{The term building block refers to  smooth functions or vector fields which are fixed before iteration, 
	and which are used to construct perturbations in a convex integration iterative scheme. 
}. Each of these building blocks then interacts {\em only} with one  term of the perturbation to the vector field (see the key identity \eqref{eqn:formulone_prodotto}). 
%We use the convex integration scheme for the continuity equation introducing two new ideas in the scheme to prove Theorem \ref{p_comp_supp}. Firstly, we keep track of a fixed number of densities\footnote{Usually the known convex integration scheme construct just one {\em wild} solution for a vector field.} and every density is perturbed using a {\em different} family of functions, called {\em building blocks}, each of them interacts {\em only} with a different part of the perturbation of the vector field.
  Secondly, we localize in space the corrector parts of the perturbation to the densities (see \eqref{eqn:localize_space}) which  will be negative, in order to preserve the positivity of the solutions. 
We also note that to prove Theorem \ref{p_comp_supp} in dimension $d=2$ the ideas  of \cite[Section 7]{BDLC20} need to be adapted for technical reasons.  This will be explained in Section \ref{section_dimension2}. 

\section{Preliminary lemmas}
In this section, we gather some useful lemmas from \cite{BDLC20,ML18}.  We will write $\T^d$ for $\R^d/ \Z^d$. 

 %and all the norms will be considered to be uniform in time.

\begin{lemma}\label{l_geom}
	Let $d,N \in \N^+$. Then, there exist disjoint families, $\Lambda^i$, of finite sets $\{ \xi\}_{\xi \in \Lambda^i} \subseteq \partial B_1\cap \mathbb{Q}^d$ for $i=1,..,N$ and smooth nonnegative coefficients $a_\xi(R)$ such that for every $R \in \partial B_1$ 
	\[
	R= \sum_{\xi \in \Lambda^i} a_\xi(R) \xi\, 
	\]
for any $i=1,..,N$.	
\end{lemma}

\subsection{Antidivergences}
We recall that the operator $\nabla \Delta^{-1}$ is an antidivergence when applied to smooth vector fields of zero mean. The following lemma proven in \cite[Lemma 2.3]{ML18} and \cite[Lemma 3.5]{MS20}  gives an improved antidivergence operator for functions with a particular structure.

\begin{lemma}\label{lemma23}(Cp. with \cite[Lemma 3.5]{MS20})
	Let $\lambda \in \N$ and $f, g : \T^d \to \R$ be smooth functions, and $g_\lambda= g(\lambda x)$. Assume that $\int g = 0$. Then if we set $\RR (f g_\lambda) = f \nabla \Delta^{-1} g_\lambda -\nabla \Delta^{-1} (\nabla f \cdot \nabla \Delta^{-1}g_\lambda+\int fg_\lambda)$, we have that $\diver  \RR (f g_\lambda) = f g_\lambda-\int fg_\lambda$ and for some $C:=C({k,p})$
	\begin{equation}
	\label{ts:antidiv}
	\|D^k \RR (f g_\lambda)\|_{L^p} \leq C \lambda^{k-1} \|f\|_{C^{k+1}} \| g\|_{W^{k,p}} \qquad \mbox{for every } k\in \N, p\in [1,\infty].
	\end{equation}
\end{lemma}
\begin{proof}
	It is enough to combine \cite[Lemma 3.5]{MS20} and the remark in \cite[page 12]{MS20}.
\end{proof}

\subsection{Slow and fast variables}
Finally we recall the following improved H\"older inequality, stated as in \cite[Lemma 2.6]{ML18} (see also \cite[Lemma 3.7]{BV18}).
If $\lambda \in \N$ and $f,g:\T^d \to \R$ are smooth functions, then we have 
\begin{equation}
\label{eqn:impr-holder}
\| f(x) g(\lambda x) \|_{L^p} \leq \| f \|_{L^p} \| g \|_{L^p} + \frac{C(p)\sqrt d \|f \|_{C^1} \|g\|_{L^p}}{\lambda^{1/p}}
\end{equation}
and 
%If moreover $\int g= 0$, then
\begin{equation}
\label{eqn:l26}
\Big| \int f(x) g(\lambda x) \, dx \Big| \leq\Big| \int f(x) \Big(g(\lambda x) - \int g \Big) \, dx \Big| + \Big| \int f \Big |\cdot \Big| \int g \Big| \leq \frac{\sqrt d \|f \|_{C^1} \|g\|_{L^1}}{\lambda} + \Big| \int f \Big| \cdot \Big | \int g \Big| .
\end{equation}

\subsection{Building blocks} \label{sec:building}
The building blocks are the same as those of \cite[Section 4]{BDLC20}. We recall them here for the convenience of the reader. 

Let $0<\rho<\frac 14$ be a constant%, which will be chosen depending only on the dimension
. We consider $\varphi\in C^\infty_c (B_\rho)$ and $\psi\in C^\infty_c (B_{2\rho})$ which satisfy
%\begin{itemize}
	$$\int \varphi =1, \qquad \varphi \geq 0, \qquad
	\psi \equiv 1 \mbox{ on }B_\rho.$$
%\end{itemize}
Given $\mu\ll1$ we define the $1$-periodic functions
\begin{align*}
\bar \varphi_\mu (x) &:= \sum_{k\in \mathbb Z^d} \mu^{d/p} \varphi (\mu (x+k))\\
\bar \psi_\mu (x) &:= \sum_{k\in \mathbb Z^d} \mu^{d/p'} \psi (\mu (x+k))\, .
\end{align*}
Let $\omega: \mathbb R^d \to \mathbb R$ be a  smooth $1$-periodic function such that $\omega(x)  = x\cdot \xi'$ on $B_{2\rho}(0)$.

Given $\Lambda$ as in Lemma~\ref{l_geom}, for any $\xi\in \Lambda$ we chose $\xi'\in \partial B_1$ such that $\xi\cdot \xi'=0$ and we define
\[
\Omega_{\xi}^\mu(x):= \mu^{-1} \omega(\mu\, x) ( \xi\otimes \xi' - \xi'\otimes \xi ).
\]
Notice that $\diver \Omega^\mu_{\xi}$ is divergence-free since $\Omega_{\xi}^\mu$ is skew-symmetric and $\diver \Omega_{\xi}^\mu=\xi$ on $\supp (\bar \psi_\mu)$ and $\supp (\bar \varphi_{\mu})$.

For $\sigma>0$ we set 
\begin{align} \label{d_building_apriori}
\tilde W_{\xi, \mu, \sigma} (t,x) &:= \sigma^{1/p'}  \diver \left[ (\Omega_{\xi}^\mu \bar \psi_\mu) (x- \mu^{d/p'} \sigma^{1/p'} t \xi )\right]\\
\tilde \Theta_{\xi, \mu, \sigma} (t,x) &:= \sigma^{1/p} \bar \varphi_\mu (x - \mu^{d/p'} \sigma^{1/p'} t \xi )\, .
\end{align}
Notice that $\tilde W_{\xi, \mu, \sigma} $ is divergence-free since it is also the divergence of the skew-symmetric matrix $\Omega_{\xi}^\mu \bar \psi_\mu$. By construction we have 
\[
\tilde W_{\xi, \mu, \sigma} (t,x)=\sigma^{1/p'}\left[ \bar \psi_{\mu} \xi + \Omega_{\xi}^{\mu} \cdot \nabla \bar \psi_{\mu}  \right](x-\mu^{d/p'} \sigma^{1/p'} t \xi ),
\]
hence the following properties are easily verified.
If we consider the translations
$$
W_{\xi, \mu, \sigma}(t, x) =\tilde W_{\xi, \mu, \sigma}(t, x-v_\xi), \qquad
\tilde\Theta_{\xi, \mu, \sigma} (t,x) = \tilde\Theta_{\xi, \mu, \sigma}(t, x-v_\xi),
$$
we can prove the following result.

\begin{lemma} \label{l_buildingblocks}
Let $d \geq 3$, $\Lambda \subset \partial B_1 \cap \Q^d$ be a finite set. Then there exists $\mu_0 >0$ such that the following holds.

 There exist two families of functions
$\{ \Theta_{\xi, \mu,\sigma} \}_{\xi, \mu, \sigma} \subset C^\infty(\T^d)$, $\{ W_{\xi, \mu,\sigma} \}_{\xi, \mu, \sigma} \subset C^\infty(\T^d;\R^d)$, where $\xi \in \Lambda, \sigma, \mu \in \R$ such that 
 for any  $\mu \geq \mu_0$, $\sigma >0$  we have
 
 \begin{equation}\label{eqn:itsolves}
	\partial_t \Theta_{\xi, \mu, \sigma}+\diver ( W_{\xi, \mu, \sigma} \Theta_{\xi, \mu, \sigma})=0,
	\end{equation}
 $$\diver W_{\xi, \mu, \sigma}=0,$$
 \begin{equation}
 \label{eqn:avW}
 \int  W_{\xi, \mu, \sigma}=0,
 \end{equation}
	\begin{equation}\label{eqn:rightaverage}
	\int  W_{\xi, \mu, \sigma} \Theta_{\xi, \mu, \sigma}= \sigma \xi.
	\end{equation}
For any $k\in \N$ and any $s\in [1,\infty]$ one has 
\begin{equation}\label{eqn:Thetanorms}
\| D^k \Theta_{\xi, \mu, \sigma} \|_{L^s}\le C(d,k,s) \sigma^{1/p} \mu^{k + d ( 1/p - 1/s )},
\qquad
\|\partial_t^k \Theta_{\xi, \mu, \sigma}\|_{L^s} \le C(d,k,s) \sigma^{1 + \frac{k - 1}{ p' }} \mu^{k + d ( \frac{ k - 1 }{ p' } + 1 - \frac{1}{s} ) }
\end{equation}

\begin{equation}\label{eqn:Wnorms}
\| D^k  W_{\xi,\mu,\sigma} \|_{L^s} \le C(d,k,s) \sigma^{1/p'} \mu^{ k+ d(1/p'-1/s)},
\qquad
\| \partial_t^k   W_{\xi,\mu,\sigma} \|_{L^s}\le C(d,k,s) \sigma^{\frac{k + 1}{ p' }} \mu^{ k + d(\frac{k+1}{p'}-\frac{1}{s} ) }.
\end{equation}
Finally, they have pairwise compact disjoint supports for any $\xi \neq \xi'$, namely
\begin{align} \label{eqn:disjointsupp}
{\supp} W_{\xi,\mu,\sigma}   \cap {\supp} \Theta_{\xi',\mu,\sigma} = {\supp} W_{\xi,\mu,\sigma} \cap {\supp} W_{\xi',\mu,\sigma} = {\supp} \Theta_{\xi',\mu,\sigma} \cap {\supp} \Theta_{\xi,\mu,\sigma} = \emptyset,
\end{align}
for any $\xi \neq \xi'$.
\end{lemma}

The proof of the previous lemma follows combining \cite[Lemma $4.1$ and Lemma $4.2$]{BDLC20}.

\section{Iteration scheme}\label{section_iteration}
 The convex integration scheme to construct solutions of the continuity equation was first introduced in \cite{ML18}. The scheme was later adapted  in \cite{BDLC20} to construct positive solutions of the continuity equation. To prove Theorem \ref{t_main} we will adapt the convex integration scheme of \cite{BDLC20}. 

First, we define the notion of a family of $a$-open sets. This notion is useful because of the convolution step in the iteration scheme. 

\begin{definition}
Let $N \in \N$, $a \in \R^+$ and $\{A_i \}_{1 \leq i \leq N}$ be a finite family of open sets of $\T^d$. We say that the family $\{A_i \}_{1 \leq i \leq N}$ is $a$-open if for any $i =1,\dots,N$ there exists a ball $B_i$ of radius $\frac{1}{4aN}$ such that $B_i \subset A_i $.
\end{definition}
As in \cite{ML18} we consider the following system of equations in $[0,1] \times \T^d$, where $d \geq 2$,
\begin{equation}\label{eqn:CE-R}
\begin{cases}
\partial_t \rho_{q,i} + \diver (  \rho_{q,i} u_{q})= - \diver R_{q,i}
\\ \\
\diver u_q=0,
\end{cases}
\end{equation}
where the indices are $i,q \in \N$ and $1 \leq i \leq N$.
We then fix three parameters $a_0$, $b>0$ and $\beta>0$, to be chosen later only in terms of $d$, $p$, $r$, and for any choice of $a>a_0$ we define
\[
\lambda_0 =a, \quad \lambda_{q+1} = \lambda_q^b \quad \mbox{and}\quad
\delta_{q} = \lambda_{q}^{-2\beta}\, .
\]
The following proposition builds a converging sequence of functions with the inductive estimates
\begin{equation}
\label{eqn:ie-1}
\max_t \|R_{q,i} (t, \cdot)\|_{L^1} \leq \delta_{q+1}
\end{equation}
\begin{equation}
\label{eqn:ie-2}
\max_t \left(\| \rho_{q,i} (t, \cdot)\|_{C^1} +\|\partial_t \rho_{q,i} (t, \cdot)\|_{C^0} + \|u_q (t, \cdot)\|_{W^{1,p'}} + \| u_q (t, \cdot)\|_{W^{2,r}} + \|\partial_t u_q (t, \cdot)\|_{L^1}\right) %+\| R_q\|_{C^1} 
\leq \lambda_{q}^\alpha\, ,
\end{equation}
for any $i=1,\dots,N$,
where $\alpha$ is yet another positive parameter which will be specified later.

\begin{prop}\label{p_inductive}
Let $d,N\in \N$, $d \geq 3$. There exist $\alpha,b, a_0, M>5$, $0<\beta<(2b)^{-1}$ such that the following holds. For every family $\{A_i \}_{1 \leq i \leq N}$ of $a_0$-open in $\T^d$ and for every $a\geq a_0$, if
	$\{ (\rho_{q,i}, u_q, R_{q,i}) \}_{1\leq i \leq N}$ solve \eqref{eqn:CE-R} and enjoy the estimates \eqref{eqn:ie-1}, \eqref{eqn:ie-2}, then there exist $\{(\rho_{q+1,i}, u_{q+1}, R_{q+1,i}) \}_{1\leq i \leq N}$ which solve \eqref{eqn:CE-R} and enjoy the estimates \eqref{eqn:ie-1}, \eqref{eqn:ie-2} with $q$ replaced by $q+1$. Moreover, for any $i=1,\dots,N$, the following hold:
	\begin{itemize}
		\item[(a)] $ \sup_ {[0,1]}[ \|( \rho_{q+1,i}- \rho_{q,i}) (t, \cdot)\|_{L^p}^p +\|(u_{q+1}- u_q) (t, \cdot)\|_{W^{1,r}}^r +\|(u_{q+1}- u_q) (t, \cdot)\|_{L^{p'}}^{p'} ] \leq M \delta_{q+1}$;
		\item[(b)] the following properties
		$$  \inf_{[0,1] \times (\T^d \setminus A_i)} \rho_{q,i} \geq 0, \qquad \inf_{[0,1] \times A_i} \rho_{q,i} \geq c >0, $$
		imply 
		$$  \inf_{[0,1] \times (\T^d \setminus A_i)} \rho_{q+1,i} \geq 0, \qquad \inf_{[0,1] \times A_i} \rho_{q+1,i} \geq c- \delta_{q+1}, $$
		%\item[(c)] if for some $t_0>0$ we have that $\rho_{q,i}(t, \cdot) = 1$, $R_{q,i}(t, \cdot)=0$ for every $i=1,..,N$ and $u_q(t, \cdot)=0$ for every $t\in [0,t_0]$, then $\rho_{q+1,i}(t, \cdot) = 1$, $R_{q+1,i}(t, \cdot)=0$ for every $i=1,..,N$ and $u_{q+1}(t, \cdot)=0$ for every $t\in [0,t_0- \lambda_q^{-1-\alpha}]$.
		\item[(c)] if for some $t_0>0$ we have that $\rho_{q,i}(t, \cdot) = 1$, $R_{q,i}(t, \cdot)=0$ and $u_q(t, \cdot)=0$ for every $t\in [0,t_0]$, then $\rho_{q+1,i}(t, \cdot) = 1$, $R_{q+1,i}(t, \cdot)=0$ and $u_{q+1}(t, \cdot)=0$ for every $t\in [0,t_0- \lambda_q^{-1-\alpha}]$,
		\item [(d)] if for some $t_0>0$ we have that $\supp \rho_{q,i} (t, \cdot) \subset B_i$, $R_{q,i}(t, \cdot)=0$ and $u_q(t, \cdot)=0$ for every $t\in [t_0,1]$, then $\supp \rho_{q+1,i}(t, \cdot)  \subset B_{i, \lambda_q^{-1- \alpha}}$, $R_{q+1,i}(t, \cdot)=0$ and $u_{q+1}(t, \cdot)=0$ for every $t\in [t_0+ \lambda_q^{-1-\alpha},1]$,
		where $ B_{i, \lambda_q^{-1- \alpha}} : = \{ x \in \T^d: \dist (x, B_i) < \lambda_{q}^{-1-\alpha}  \}$.
	\end{itemize}
\end{prop}

\begin{remark}
We highlight that the constant $a_0$ in the proposition above does not depend on the sets $A_i$ 
%for any $i=1,\dots,N$, 
but only on the number $N$. 
Therefore, when we apply this proposition (to prove Theorem \ref{p_comp_supp}), we choose the sets $A_i$ after having fixed $a_0$.
%This is a detail that will allows us to choose the sets $A_i$ for the initial condition after having fixed the constant $a_0$.\red{The above sentence is not clear. Also, if you highlight something, it means that it is important, and hence not a detail.}
\end{remark}

To prove Proposition \ref{p_inductive} we use a convex integration scheme similar to the one in \cite[Proposition 2.1]{BDLC20}. However, the end products of our scheme are different from those of \cite{BDLC20}. Indeed, we seek to produce a \emph{single} vector field $u$ and $N$ densities $\rho_i$ with mutually disjoint compact supports for some time such that $(\rho_i, u)$ weakly solves \eqref{eq_continuity}. Accordingly, we modified the iterative proposition of \cite{BDLC20} in two essential ways: we index $N$ distinct densities $\rho_{q,i}$ by the parameter $i=1,\dots, N$; we have refined the control $\inf \rho_{q+1,i}$ over the subregion $\T^d\setminus A_i$ thanks to (b). The former is achieved  by taking $N$ disjoint families of building blocks $\{\Lambda_i\}_{1\leq i\leq N}$, and the latter by localizing to $A_i$ the corrector part of the perturbation to $\rho_{q,i}$.

\subsection{Choice of the parameters}\label{sec:param}
The choice of parameter is the same of \cite[Section 5.1]{BDLC20}.
We define first the constant
\[
\gamma:= \Big(1+\frac 1p\Big) \left(\min \Big\{ \frac{d}{p}, \frac{d}{p'}, -1-d  \Big(\frac1{p'}-\frac1r\Big)\Big\}\right)^{-1}>0,
\]
Notice that, up to enlarging $r$, we can assume that the quantity in the previous line is less than $1/2$, namely that $\gamma>2$.
Hence we set $\alpha:=4 + \gamma (d+1)$,
\begin{equation}
\label{eqn:choice-b}
b := \max\{ p, p'\} (3(1+\alpha)(d+2)+2),
\end{equation}
and 
\begin{equation}
\label{eqn:choice-beta}
\beta:= \frac 1{2 b} \min\Big\{ p, p', r, \frac{1}{b+1}\Big\}= \frac{1}{2 b(b+1)} .
\end{equation}
Finally, we choose $a_0$ and $M$ sufficiently large (possibly depending on all previously fixed parameters) to absorb numerical constants in the inequalities.
We set
\begin{equation}
\label{eqn:choice-ell}
\ell: = \lambda_{q}^{-1-\alpha},
\end{equation}
\begin{equation}
\label{eqn:choice-mu}\mu_{q+1} := \lambda_{q+1}^{\gamma}.
\end{equation}

\subsection{Convolution}
The convolution step is the same of \cite[Section 5.2]{BDLC20}. We just write here the definitions.
We first perform a convolution of $\rho_q$ and $u_q$ to have estimates on more than one derivative of these objects and of the corresponding error. Let $\phi \in C^\infty_c(B_1)$ be a standard convolution kernel in space-time, $\ell$ as in \eqref{eqn:choice-ell} and define 
$$\rho_{\ell,i} := \rho_{q,i} \ast \phi_\ell, \qquad u_\ell := u_q \ast \phi_\ell, \qquad R_{\ell,i} := R_{q,i} \ast \phi_\ell. $$
We observe that $(\rho_{\ell,i}, u_\ell, R_{\ell,i}+ (\rho_{q,i} u_q)_\ell - \rho_{\ell,i} u_\ell)$ solves system \eqref{eqn:CE-R} for any $i =1,\dots,N$ 
and  by \eqref{eqn:ie-1}, \eqref{eqn:choice-beta} enjoys the following estimates

\begin{equation}
\label{eqn:r-l-l1}
\|R_{\ell,i} \|_{L^1} \leq \delta_{q+1},
\end{equation}
\begin{equation} \label{eqn:rho-ell-conv}
\| \rho_{\ell,i} - \rho_{q,i}\|_{L^p} \leq  \ell \|\rho_{q,i}\|_{C^1} \leq  \ell \lambda_q^\alpha  \leq \lambda_q^{-1} \leq  \frac{\delta_{q+1}^{1/p}}{2},
\end{equation}
$$\| u_\ell - u_q\|_{L^{p'}} \leq C \ell \lambda_q^\alpha \leq C \delta_{q+1}^{1/{p'}},$$
$$\| u_\ell - u_q\|_{W^{1,r}} \leq C \ell \lambda_q^\alpha \leq C \delta_{q+1}^{1/{r}}\, .$$
Indeed note that by \eqref{eqn:choice-beta}
\[
\ell \lambda_q^\alpha = \lambda_q^{-1} = \delta_{q+1}^{\frac{1}{2b\beta}} \leq
\delta_{q+1}^{\max\{1/p, 1/p', 1/r\}}.
\]
Next observe that
\[
\|\partial_t^S \rho_{\ell,i} \|_{C^0} + \| \rho_{\ell,i} \|_{C^S} +\| u_\ell\|_{W^{1+S,r}} + \|\partial_t^S u_\ell\|_{W^{1,r}} %+\| R_{\ell,i}\|_{C^S} 
\leq  C(S) \ell^{-S+1}( \| \rho_{q,i}\|_{C^1} +\| u_q\|_{W^{2,r}} %+\| R_q\|_{C^1}
)\leq  C(S) \ell^{-S+1}\lambda_{q}^\alpha
\]
for every $S\in \N\setminus\{0\}$ and for every $i=1,\dots,N$.
Using the Sobolev embedding $W^{d,r}\subset {W^{d,1}\subset }C^0$ we then conclude
\[
\|\partial_t^S u_\ell\|_{C^0} + \|u_\ell\|_{C^S} \leq C (S) \ell^{-S-d+2} \lambda_q^\alpha\, .
\]
By Young's inequality we estimate the higher derivatives of $R_{\ell,i}$ in terms of $\|R_{q,i} \|_{L^1}$ to get
\begin{equation}\label{e:D_tR}
\| R_{\ell,i}\|_{C^S} +\|\partial_t^S R_{\ell,i} \|_{C^0} %\leq \|D^S \rho_{\ell,i} \ast R_{q,i} \|_{L^\infty} 
\leq \|D^S \rho_{\ell ,i} \|_{L^\infty} \|R_{q,i} \|_{L^1} \leq  C(S) \ell^{-S-d} \leq  C(S) \lambda_{q}^{(1+\alpha)(d+S)}\, 
\end{equation}
for every $L \in \N$ and $i =1,\dots,N$. %An analogous argument yields the estimate
%\begin{equation}\label{e:D_tR}
%\|\partial_t^N R_{\ell,i}\|_{C^0} \leq C (N) \ell^{-N-d} \leq C (N) \lambda_q^{-(1+\alpha) (d+N)}
%\end{equation}
Finally, thanks to \cite[Lemma 5.1]{BDLC20} for the last part of the error we have
\begin{equation}\label{e:commutatore}
\|(\rho_{q,i} u_q)_\ell - \rho_{\ell,i} u_\ell\|_{L^1} \leq C \ell^2 \lambda_q^{2\alpha} \leq \frac{1}{4} \delta_{q+2} ,
\end{equation}
where we have assumed that $a$ is sufficiently large.

\subsection{Definition of the perturbation}

Let $\mu_{q+1}>0$ be as in \eqref{eqn:choice-mu} and let $\chi\in C^\infty_c (-\frac{3}{4}, \frac{3}{4})$ such that $\sum_{n \in \mathbb Z} \chi (\tau - n) =1$ for every $\tau\in \mathbb R$.
Let $ \overline{\chi}\in C^{\infty}_c(-\frac{4}{5}, \frac{4}{5})$ be a nonnegative function satisfying $\overline{\chi} =1$ on $[-\frac{3}{4}, \frac{3}{4}]$.
Notice that $\sum_{n\in \mathbb{Z}} \overline{\chi} (\tau-n)\in [1,2]$ and  $\chi\cdot \overline{\chi}= \chi$.

Fix a parameter $\kappa={40 p'%\no
}/{\delta_{q+2}}$ and consider  $2N$ disjoint families $\{\Lambda^1_i\}_{1 \leq i \leq N}$, $\{\Lambda^2_i\}_{1 \leq i \leq N}$ as in Lemma \ref{l_geom}. Next, for $n \in \N$, define $[n]$ to be $1$ or $2$ depending on the congruence class of $n$. Finally, we take our building blocks according to Lemma \ref{l_buildingblocks} with $ \Lambda= \cup_{i=1}^2 \cup_{j=1}^N \Lambda^i_j$ and observe that their spatial supports are disjoint.
We define the new density and vector field by adding to $\rho_\ell$ and $u_\ell$  principal terms and 
% \red{I don't think smaller is a good adjective here.} smaller 
correctors, namely we set
\begin{align*}
\rho_{q+1,i} &:= \rho_{\ell,i} + \theta_{q+1,i}^{(p)}+ \theta_{q+1,i}^{(c)}\,,\\ 
u_{q+1}  &:= u_\ell+ \sum_{i=1}^N (w_{q+1,i}^{(p)}+ w_{q+1,i}^{(c)} ) \  .
\end{align*}
The principal perturbations are given, respectively, by
\begin{align}
\label{eqn:theta}
w^{(p)}_{q+1,i} (t,x) &= \sum_{n \geq 12}\overline{\chi} (\kappa |R_{\ell,i} (t,x)| -n) \sum_{\xi\in \Lambda_i^{[n]}} W_{\xi , \mu_{q+1}, n /\kappa} (\lambda_{q+1} t, \lambda_{q+1} x),  \\ 
\label{eqn:w}
\theta^{(p)}_{q+1,i} (t,x) &= \sum_{n \ge 12} \chi (\kappa |R_{\ell,i} (t,x)| - n) \sum_{\xi\in \Lambda_i^{[n]}}  a_{\xi}\left(\frac{R_{\ell,i}(t,x)}{|R_{\ell,i}(t,x)|}\right) \Theta_{\xi, \mu_{q+1}, n /\kappa} (\lambda_{q+1} t, \lambda_{q+1} x)\, ,
\end{align} 
where it is understood that the terms in the second sum vanish at points where $R_{\ell,i}$ vanishes. 
%\begin{remark}\label{remark:sum is finite}
	In the definition of $w^{(p)}_{q+1,i}$ and $\theta^{(p)}_{q+1,i}$ the first sum runs for $n$ in the range 
	\begin{equation}
	\label{remark:sum is finite}
		12\leq n \leq  C\ell^{-d}\delta_{q+2}^{-1} \le C \lambda_{q}^{d(1+\alpha) + 2\beta b^2}
	\leq C \lambda_{q}^{d(1+\alpha)+1} \leq \lambda_{q}^{d(1+\alpha)+2},
	\end{equation}
	where the last holds providing $a_0\geq C$.
	Indeed $\chi(\kappa |R_{\ell,i}(t,x)|-n)=0$ if $n\ge 20%\no
	\delta_{q+2}^{-1}\|R_{\ell,i}\|_{C^0}+1$ and by \eqref{e:D_tR} we obtain an upper bound for $n$.
%\end{remark}

The aim of the corrector term for the density is to ensure that the overall perturbation has zero average. So we set
\begin{equation} \label{eqn:localize_space}
\theta^{(c)}_{q+1,i} (t,x):=- g_i(x) \int \theta_{q+1,i}^{(p)}(t,x)\, dx,
\end{equation}
where $\{g_i\}_{1 \leq i \leq N} \subset C^\infty(\T^d)$ such that $\int_{\T^d} g_i =1$, $g_i \geq 0$ and supp$(g_i)$ is compactly contained on $A_i$ and $\| g_i \|_{C^S} \lesssim (\lambda_0 N)^{d+S} $ (where $\lesssim$ means inequality up to a geometric constant depending only on $d$). Here we used the property that the family $\{A_i\}_{1 \leq i \leq N}$ is $\lambda_0$-open.  
%We observe that these functions,  $\{g_i\}_{i}$, are not depending on $q$, they depend only on the fixed open sets $\{A_i \}_{1 \leq i \leq N}$.
We observe that the functions  $\{g_i\}_{1\leq i\leq N}$ do not depend on $q$, they only depend only on the fixed open sets $\{A_i \}_{1 \leq i \leq N}$.
The aim of the corrector term for the vector field is to ensure that the overall perturbation has zero divergence. Thanks to \eqref{eqn:avW}, we can apply Lemma~\ref{lemma23} to define

\[
w_{q+1,i}^{(c)}:=-\sum_{n \geq 12} \sum_{\xi \in \Lambda_i^{[n]}} \mathcal{R} 
\left[ \nabla \overline{\chi}(\kappa |R_{\ell,i}(t,x)|-n) \cdot W_{\xi,\mu_{q+1},n/\kappa}(\lambda_{q+1} t, \lambda_{q+1} x) \right]
\]
Moreover, since $W_{\xi,\mu_{q+1},n/\kappa}$ is divergence-free, the argument inside $\mathcal{R}$ has $0$ average for every $t\geq 0$, and so $w_{q+1,i}^{(c)}$ is indeed well defined. 
Notice finally that the perturbations equals $0$ on every time interval  where $R_{\ell,i}$ vanishes identically for any $i=1,\dots,N$.

\section{Proof of Proposition \ref{p_inductive}}

For the sake of readability, the quantifier ``for every $i=1,\dots, N$'' will be implicit in the rest of this paper. 
Before coming to the main arguments, we recall \cite[Lemma 6.1]{BDLC20}  for the ``slowly varying coefficients''.

\begin{lemma}\label{l:uglylemma}
	For $m\in \mathbb N$, $S\in \mathbb N\setminus \{0\}$ and $n\ge 2$ we have
	\begin{align*}
	&\|\partial^m_t \chi (\kappa|R_{\ell,i}|-n)\|_{C^S}+\|\partial^m_t\bar \chi (\kappa|R_{\ell,i}|-n)\|_{C^S} \leq C (m, S) \delta_{q+2}^{-2(S+m)} \ell^{-(S+m)(1+d)} 
	\leq 
	C(m,S) \lambda_q^{(S+m) (d+2) (1+\alpha)}\\
%	{\color{blue}|\partial^m_t\bar \chi (\kappa|R_{\ell,i}|-n)\|_{C^N}}& \leq C (m, N) \delta_{q+2}^{-2(N+m)} \ell^{-(N+m)(1+d)}
%	\leq 
%	C(m,N) \lambda_q^{(N+m) (d+2) (1+\alpha)}\\
	&\|\partial^m_t (a_\xi ({\textstyle{\frac{R_{\ell,i}}{|R_{\ell,i}|}}}))\|_{C^S}\leq C (m, S) \delta_{q+2}^{-S-m} \ell^{-(S+m)(1+d)} 
	\leq C(m,S)\lambda_q^{(S+m) (d+2) (1+\alpha)} \qquad\mbox{on $\{\chi (\kappa|R_{\ell,i}|-n) >0\}$}.
	\end{align*}
\end{lemma}

\subsection{Estimate on $\|\theta_{q+1,i}\|_{L^p}$ and on $\inf_{\T^d} \theta_{q+1,i}$}

We apply the improved H\"older inequality of \eqref{eqn:impr-holder},  Lemma \ref{l:uglylemma} and \eqref{remark:sum is finite} to get
\begin{equation}\label{eqn:theta-pert-p}
\begin{split}
\|\theta_{q+1,i}^{(p)}\|_{L^p} &\leq \sum_{n \geq 12} \sum_{\xi\in \Lambda_i^{[n]}} \| \chi (\kappa |R_{\ell,i} (t,x)| - n)  \textstyle{a_{\xi}\left(\frac{R_{\ell,i}(t,x)}{|R_{\ell,i}(t,x)|}\right)} \|_{L^p} \|\Theta_{\xi, \mu_{q+1}, n /\kappa} (\lambda_{q+1} t, \lambda_{q+1} x) \|_{L^p}\\
\quad +& \frac{1}{\lambda_{q+1}^{1/p}} \sum_{n\geq 12} \sum_{\xi\in \Lambda^{[n]}}  \| \chi (\kappa |R_{\ell,i} (t,x)| - n)  \textstyle{a_{\xi}\left(\frac{R_{\ell,i}(t,x)}{|R_{\ell,i}(t,x)|}\right)} \|_{C^1} \|\Theta_{\xi, \mu_{q+1}, n /\kappa} (\lambda_{q+1} t, \lambda_{q+1} x) \|_{L^p}
\\& \le C\sum_{n\geq 12} \| (n/k)^{1/p} \chi (\kappa |R_{\ell,i} (t,x)| - n)  \|_{L^p} 
+C \lambda_{q+1}^{-1/p} \delta_{q+2}^{1/p} \lambda_q^{(d+2)(1+\alpha)+(1+1/p)((d(1+\alpha)+1))}
\\& \le C  \|R_{\ell,i}\|_{L^1}^{1/p} + C \lambda_{q+1}^{-1/p} \delta_{q+2}^{1/p} \lambda_q^{3(d+2) (1+\alpha)} %C\delta_{q+2}^{1/p}
\\& 
\leq C \delta_{q+1}^{1/p},
\end{split}
\end{equation}
provided that in the second last inequality we use $$(d+2)(1+\alpha)+(1+1/p)((d(1+\alpha)+1))\le 3(1+\alpha)(d+2)$$ and in the last inequality we use \eqref{eqn:r-l-l1}.
%\footnote{EB: qui ho usato .}

Next , we use 
$$\| \Theta_{\xi, \mu_{q+1}, n/ \kappa}\|_{L^1}\le (\frac{n}{\kappa})^{1/p}\mu_{q+1}^{-d/p'}$$
from \eqref{eqn:Thetanorms}, 
\eqref{eqn:l26} applied to the $\lambda_{q+1}^{-1}$-periodic function $(\Theta_{\xi, \mu_{q+1}, n/\kappa}(\lambda_{q+1}t, \lambda_{q+1}x)$, 
\eqref{remark:sum is finite} and Lemma \ref{l:uglylemma} to get

\begin{align} \label{eqn:theta_c}
| \theta_{q+1,i}^{(c)} (t,x)| 
& \le  \sqrt{d} \lambda_{q+1}^{-1} \sum_{n\geq 12}\sum_{\xi\in \Lambda^{[n]}} \|g_i\|_{L^\infty} \|  \chi( \kappa|R_{\ell,i}| - n ) \textstyle{a_{\xi}\left( \frac{ R_{\ell,i} }{ |R_{\ell,i}| } \right)} \|_{C^1} 
\| \Theta_{\xi, \mu_{q+1}, n/\kappa}\|_{L^1} \notag
\\&\qquad + \sum_{n\geq 12}\sum_{\xi\in \Lambda^{[n]}} \|g_i\|_{L^\infty} \| \chi( \kappa|R_{\ell,i}| - n ) \textstyle{a_{\xi}\left( \frac{ R_{\ell,i} }{ |R_{\ell,i}| } \right)} \|_{L^1} \| \Theta_{\xi, \mu_{q+1}, n/\kappa}\|_{L^1}  \notag
\\& \le \| g_i \|_{L^\infty} \lambda_{q+1}^{-1} \mu_{q+1}^{-d/p'} \lambda_{q}^{3(1+\alpha)(d+2)}
+ \| g_i \|_{L^\infty}  \mu_{q+1}^{-d/p'} \sum_{n\geq 12}  (n/\kappa)^{1/p} \| \chi( \kappa|R_{\ell,i}| - n ) \textstyle{a_{\xi}\left( \frac{ R_{\ell,i} }{ |R_{\ell,i}| } \right)} \|_{L^1} \notag
\\& \le \| g_i \|_{L^\infty} \mu_{q+1}^{-d/p'} +  \| g_i \|_{L^\infty} \mu_{q+1}^{-d/p'}\|R_{\ell,i} \|_{L^1}^{1/p} \lesssim (N \lambda_0)^d \lambda_{q+1}^{-1} \le  \frac{ \delta_{q+1}}{2},
\end{align} 
where in the second to last inequality, we enlarge $a$ to absorb the constant $N$.
Now, recall that $\theta_{q+1}^{(p)}$ is nonnegative by definition. Therefore, 
\begin{align*}
\underset{[0,1]\times \T^d}{\inf } [\theta_{q+1,i}^{(p)}+\theta_{q+1,i}^{(c)} ]\geq \underset{[0,1]\times \T^d}{\inf } \theta_{q+1,i}^{(c)} \geq - \frac{\delta_{q+1}}{2}.
\end{align*}
Also 
\begin{align*}
\underset{[0,1]\times (\T^d\setminus A_i)}{\inf} [\theta_{q+1,i}^{(p)}+\theta_{q+1,i}^{(c)} ] \geq 0.
\end{align*}
%which in view of  \eqref{eqn:rho-ell-conv} and that $\rho_{\ell,i}$ is non negative if $\rho_{q,i}$ is non negative, gives property (b) of Proposition \ref{p_inductive}.
Since $\rho_{\ell,i}$ is nonnegative whenever $\rho_{q,i}$ is nonnegative, by \eqref{eqn:rho-ell-conv} we get property (b) of Proposition \ref{p_inductive}. 

\begin{comment}
Since $\theta^{(c)}_{q+1,i}$ is supported on $[0,1]\times A_i$, we have, on $\T^d\setminus A_i$, that 
\begin{align*}
\underset{[0,1]\times (\T^d\setminus A_i)}{\inf} [\theta_{q+1,i}^{(p)}+\theta_{q+1,i}^{(c)} ] \geq 0,
\end{align*}
and up to enlarging the set $A_i$ to $\{x\in \T^d: \dist(x,A_i)\leq l\}$, we have property (b) of Proposition~\ref{p_inductive}.
\end{comment}
%Moreover from the hypothesis (b) of Proposition \ref{p_inductive} on $\rho_{q,i}$ we also have
%$\rho_{\ell,i} \geq 0$ that, together with the previous properties, implies Statement (b') of Proposition \ref{p_inductive}.
\subsection{Estimate on $\|w_{q+1,i}\|_{L^{p'}}$ and $\|Dw_{q+1,i}\|_{L^{r}}$} \label{sec:end}

Exactly with the same computation as in \eqref{eqn:theta-pert-p}, replacing $p$ with $p'$, we have that
\begin{equation}\label{eq:wp:lp'}
\|w_{q+1,i}^{(p)}\|_{L^{p'}} 
\leq C \|R_{\ell,i}\|_{L^1}^{1/p'}+
C \lambda_{q+1}^{-1/p'} \delta_{q+2}^{1/p'} \lambda_q^{3(d+2) (1+\alpha)}
\leq C \delta_{q+1}^{1/p'}
\end{equation}
Concerning the corrector term $w_{q+1}^{(c)}$, we use \eqref{eqn:Wnorms} (precisely $\| W_{\xi, \mu_{q+1}, n/ \kappa}\|_{L^{p'}}\le (\frac{n}{\kappa})^{1/p'}\le \lambda_{q}^{(d+2)(1+\alpha)}
$) Lemma~\ref{lemma23} and \eqref{remark:sum is finite} to get
\begin{equation}\label{eqn:wcpprime}
\begin{split}
\|w_{q+1,i}^{(c)}\|_{L^{p'}}  &\leq
\frac 1 {\lambda_{q+1}} \sum_{n\ge \nop}\sum_{\xi \in \Lambda_i^{[n]}} 
\|\bar \chi(\kappa|R_{\ell,i}|-n)\|_{C^2}\|W_{\xi,\mu_{q+1},n/\kappa}\|_{L^{p'}}
\\& \leq  C\lambda_{q+1}^{-1}\sum_{n\geq\nop}%^{ C\lambda_q^{d(1+\alpha)+1}}
\lambda_{q}^{2(d+2)(1+\alpha)} (n/\kappa)^{1/p'}
\\&\leq C\lambda_{q+1}^{-1} \delta_{q+2}^{1/p'}\lambda_{q}^{4(d+2)(1+\alpha)}
\le \delta_{q+2}^{1/p'} \le \delta_{q+1}^{1/p'}.
\end{split}
\end{equation}
Computing the gradient of $w_{q+1}^{(p)}$ and combining Lemma \ref{l:uglylemma} with \eqref{eqn:Wnorms} we have
\begin{align}
\|Dw_{q+1,i}^{(p)}\|_{L^r} &\leq  
\sum_{n\ge \nop}\sum_{\xi\in \Lambda_i^{[n]}} \|\bar \chi(\kappa |R_{\ell,i}|-n)\|_{C^1} \| W_{\xi,\mu_{q+1},n/\kappa}\|_{L^r}
+\sum_{n\ge \nop}\sum_{\xi\in \Lambda_i^{[n]}} \lambda_{q+1} \|DW_{\xi,\mu_{q+1},n/\kappa}\|_{L^r}
\\& \leq C \delta_{q+2}^{1/p'} \lambda_{q}^{3(1+\alpha)(d+2)+2+b\gamma d(1/p'-1/r)}
+C \delta_{q+2}^{1/p'} \lambda_{q}^{b+3(1+\alpha)(d+2)+b\gamma(1+d(1/p'-1/r))}
\\& \leq  \delta_{q+2}^{1/p'}\le \delta_{q+1}^{1/r}.
\end{align}
Concerning the corrector, by Lemma~\ref{lemma23} and similar computations as above,
\begin{align}
\|Dw_{q+1,i}^{(c)}\|_{L^r} &\leq C \sum_{n\ge \nop}\sum_{\xi\in \Lambda_i^{[n]}} \| \bar\chi(\kappa |R_{\ell,i}|-n)\|_{C^3} \|W_{\xi, \mu_{q+1}, n/\kappa}\|_{L^r} 
\\& \leq C \delta_{q+2}^{1/p'} \lambda_q^{ 5(1+\alpha)(d+2)
} \mu_{q+1}^{d(\frac 1 r - \frac 1 {p'})}
\leq \delta_{q+2}^{1/p'} \lambda_{q}^{5(1+\alpha)(d+2)-b(1+1/p)}
\leq \delta_{q+1}^{1/r}.
\end{align}

\subsection{Definition of the new error $R_{q+1,i}$} \label{sec:new}

This part is the similar to \cite{BDLC20}, however here $\theta_{q+1,i}^{(c)}$ is not constant in space and we have to adapt the argument accordingly: we will pick up a new error term which we call $R^{space}$.

By definition the new error $R_{q+1,i}$ must satisfy
\begin{equation}\label{eqn:new-error}
\begin{split}
-\diver R_{q+1,i} &= \partial_t \rho_{q+1,i} + \diver(  \rho_{q+1,i} u_{q+1}) \\
&= \diver ( \theta_{q+1,i}^{(p)} w_{q+1}^{(p)} - R_{\ell,i}) + \partial_t \theta_{q+1,i}^{(p)}+\partial_t \theta_{q+1,i}^{(c)} 
\\& \quad + \diver( \theta_{q+1,i}^{(p)} u_\ell+ \rho_{\ell,i} w_{q+1} + \theta_{q+1,i}^{(p)}w_{q+1}^{(c)} )
%+\theta_{q+1,i}^{(c)}w_{q+1}^{(p)}+\theta_{q+1}^{(c)}w_{q+1}^{(c)}
\\& \quad  + \diver (( u_{\ell}+w_{q+1})  \theta_{q+1,i}^{(c)})
+ \diver ( (\rho_{q,i} u_q)_\ell - \rho_{\ell,i} u_\ell) 
\end{split}
\end{equation}
In the second equality above we have used that $(\rho_\ell, u_\ell, R_{\ell,i}+ (\rho_{q,i} u_q)_\ell - \rho_{\ell,i} u_\ell)$ solves \eqref{eqn:CE-R}.
% that $\diver u_\ell = \diver w_{q+1}=0$.

We now decompose 
\begin{align*}
\partial_t \theta_{q+1,i}^{(p)} & = \sum_{n\ge \nop}\sum_{\xi\in \Lambda_i^{[n]}} \chi(\kappa |R_{\ell,i}|- n) \textstyle{a_{\xi}\left( \frac{R_{\ell,i}}{|R_{\ell,i}|} \right)} \partial_t\left[ \Theta_{\xi, \mu_{q+1}, n /\kappa}(\lambda_{q+1} t, \lambda_{q+1}x) \right]
\\& \quad + \sum_{n\ge \nop}\sum_{\xi\in \Lambda_i^{[n]}} \partial_t \left[ \chi(\kappa |R_{\ell,i}|- n) \textstyle{a_{\xi}\left( \frac{R_{\ell,i}}{|R_{\ell,i}|} \right)}\right]  \Theta_{\xi, \mu_{q+1}, n /\kappa}(\lambda_{q+1} t, \lambda_{q+1}x)
\\&=: ( \partial_t \theta^{(p)}_{q+1,i})_1+(\partial_t \theta^{(p)}_{q+1,i})_2,
\end{align*}

We now observe that
\begin{equation} \label{eqn:formulone_prodotto}
\begin{split}
\theta_{q+1,i}^{(p)}\sum_{k=1}^N w^{(p)}_{q+1,k}&= \theta_{q+1,i}^{(p)}w^{(p)}_{q+1,i}\\
&=\sum_{n\ge \nop} \sum_{\xi\in \Lambda_i^{[n]}}    \chi(\kappa |R_{\ell,i}|- n)
\textstyle{a_{\xi}\left(\frac{R_{\ell,i}}{|R_{\ell,i}|}\right)} (\Theta_{\xi, \mu_{q+1}, n /\kappa} W_{\xi, \mu_{q+1}, n /\kappa})(\lambda_{q+1} t, \lambda_{q+1}x),
\end{split} 
\end{equation}
where the first equality holds because $\Lambda_i^k \cap \Lambda_j^p= \emptyset$ for any $p,k$ and $i \neq j$; the second equality follows from \eqref{eqn:disjointsupp} and the definitions of $\chi$ and $\overline{\chi}$. Also, $\Theta_{\xi, \mu_{q+1}, n /\kappa}$ and $W_{\xi, \mu_{q+1}, n /\kappa}$ solve the transport equation \eqref{eqn:itsolves}. These observations in conjunction with Lemma \ref{l_geom} yield the cancellation of the error $R_{\ell,i}$ up to lower order terms
\begin{equation}\label{eqn:formulone}
\begin{split}
&\diver ( \theta_{q+1,i}^{(p)}\sum_{k=1}^N w_{q+1,k}^{(p)}) + (\partial_t \theta_{q+1,i}^{(p)})_1 - \diver R_{\ell,i}
\\& =
\sum_{n\ge \nop} \sum_{\xi\in \Lambda_i^{[n]}}   \nabla \left[\chi(\kappa |R_{\ell,i}|- n)
\textstyle{a_{\xi}\left(\frac{R_{\ell,i}}{|R_{\ell,i}|}\right)}\right] (\Theta_{\xi, \mu_{q+1}, n /\kappa} W_{\xi, \mu_{q+1}, n /\kappa})(\lambda_{q+1} t, \lambda_{q+1}x)- \diver R_{\ell,i}
\\&+\sum_{n\ge \nop} \sum_{\xi\in \Lambda_i^{[n]}} \chi(\kappa |R_{\ell,i}|- n)
\textstyle{a_{\xi}\left(\frac{R_{\ell,i}}{|R_{\ell,i}|}\right)}
\lambda_{q+1} \left[ \partial_t \Theta_{\xi, \mu_{q+1}, n /\kappa}+ \diver(\Theta_{\xi, \mu_{q+1}, n /\kappa}W_{\xi, \mu_{q+1}, n /\kappa})\right](\lambda_{q+1} t, \lambda_{q+1}x)
\\&=\sum_{n\ge \nop} \sum_{\xi\in \Lambda_i^{[n]}}   \nabla \left[\chi(\kappa |R_{\ell,i}|- n)
\textstyle{a_{\xi}\left(\frac{R_{\ell,i}}{|R_{\ell,i}|}\right)}\right] \left[(\Theta_{\xi, \mu_{q+1}, n /\kappa} W_{\xi, \mu_{q+1}, n /\kappa})(\lambda_{q+1} t, \lambda_{q+1}x)- \frac{n}{\kappa} \xi \right]
\\&+ \sum_{n\ge \nop} \sum_{\xi\in \Lambda_i^{[n]}}   \nabla \left[\chi(\kappa |R_{\ell,i}|- n)
\textstyle{a_{\xi}\left(\frac{R_{\ell,i}}{|R_{\ell,i}|}\right)}\right] \frac{n}{\kappa}\xi - \diver R_{\ell,i}
\\& =\sum_{n\ge \nop} \sum_{\xi\in \Lambda_i^{[n]}}   \nabla \left[\chi(\kappa |R_{\ell,i}|- n)
\textstyle{a_{\xi}\left(\frac{R_{\ell,i}}{|R_{\ell,i}|}\right)}\right] \left[(\Theta_{\xi, \mu_{q+1}, n /\kappa} W_{\xi, \mu_{q+1}, n /\kappa})(\lambda_{q+1} t, \lambda_{q+1}x)- \frac{n}{\kappa} \xi \right]
\\&
\qquad\qquad\qquad\qquad+ \diver (\tilde R_{\ell,i}-R_{\ell,i}),
\end{split} 
\end{equation}
where
$$\tilde R_{\ell,i}:= 
\sum_{n\ge \nop} \chi(\kappa |R_{\ell,i}|- n)	\frac{R_{\ell,i}}{|R_{\ell,i}|} \frac{n}{k}.$$

We have
\begin{align}
\notag
|R_{\ell,i}-\tilde R_{\ell,i}|  
&\leq \Big|\sum_{n=-1}^{11%\no
} \chi(\kappa |R_{\ell,i}|- n) R_{\ell,i} \Big|
+\Big| \sum_{n\ge \nop} \chi(\kappa |R_{\ell,i}|- n)\left( \frac{R_{\ell,i}}{|R_{\ell,i}|}\frac{n}{k}-R_{\ell,i}\right)	\Big|   \label{eqn:tildeRestimate}
\\& \le \frac{13%\no
}{\kappa} + \sum_{n\ge \nop} \chi(\kappa |R_{\ell,i}|- n)\left| |R_{\ell,i}|-\frac{n}{\kappa}\right|\\
(\text{by definition of }\kappa\text{ and } \eqref{remark:sum is finite}) & \le \frac{13%\no
}{20%\no
}\delta_{q+2}+\frac{3}{40%\no
}\delta_{q+2}
\le \frac{15%\no
}{20%\no
} \delta_{q+2}.\notag
\end{align}

We can now define $R_{q+1,i}$ which satisfies \eqref{eqn:new-error} as 
\begin{equation}
\begin{split}
-R_{q+1,i} :=& R_i^{quadr}+ (\tilde R_{\ell,i}- R_{\ell,i}) + R_i^{time} +R_i^{space} +\theta_{q+1,i}^{(p)} u_\ell+ \rho_{\ell,i} w_{q+1} + \theta_{q+1,i}^{(p)}w_{q+1}^{(c)}+[(\rho_{q,i} u_q)_\ell - \rho_{\ell,i} u_\ell],
\end{split}
\end{equation}
where
\begin{equation}
\label{defn:Rquadr}
R_i^{quadr}: = \sum_{n\ge \nop} \sum_{\xi\in \Lambda_i^{[n]}} \mathcal{R} \left[  \nabla \left(\chi(\kappa |R_{\ell,i}|- n)
\textstyle{a_{\xi}\left(\frac{R_{\ell,i}}{|R_{\ell,i}|}\right)}\right)\cdot  \left((\Theta_{\xi, \mu_{q+1}, n /\kappa} W_{\xi, \mu_{q+1}, n /\kappa})(\lambda_{q+1} t, \lambda_{q+1}x)- \frac{n}{\kappa} \xi \right)\right],
\end{equation}
%and 
\begin{equation}
\label{eqn:Rc}
R_i^{time}:=\nabla \Delta^{-1}( (\partial_t \theta_{q+1,i}^{(p)})_2 +\partial_t \theta_{q+1,i}^{(c)} + m _i),
\end{equation}
\begin{equation*}
m _i:=\sum_{n\ge \nop} \sum_{\xi\in \Lambda_i^{[n]}} \int    \nabla \left[\chi(\kappa |R_{\ell,i}|- n)
\textstyle{a_{\xi}\left(\frac{R_{\ell,i}}{|R_{\ell,i}|}\right)}\right] \left[(\Theta_{\xi, \mu_{q+1}, n /\kappa} W_{\xi, \mu_{q+1}, n /\kappa})(\lambda_{q+1} t, \lambda_{q+1}x)- \frac{n}{\kappa} \xi \right] \, dx,
\end{equation*}
\begin{equation}
\label{eqn:Rspace}
R_i^{space}:= (u_\ell +w_{q+1}) \theta_{q+1,i}^{(c)}.
\end{equation}
Property (d) is now clear from the definition of $R_{q+1,i}$ and the definition of $\rho_{q+1,i}.$
Notice that $R_i^{quadr}$ is well defined since by \eqref{eqn:rightaverage} the function $(\Theta_{\xi, \mu_{q+1}, n /\kappa} W_{\xi, \mu_{q+1}, n /\kappa})(\lambda_{q+1} t, \lambda_{q+1}x)- \frac{n}{\kappa} \xi$ has $0$ mean. From the second equality in \eqref{eqn:new-error} and since the average of $(\partial_t \theta_{q+1,i}^{(p)})_1$ is $m_i$ by integration by parts, we deduce that $ (\partial_t \theta_{q+1,i}^{(p)})_2 +\partial_t \theta_{q+1,i}^{(c)} + m_i$ has $0$ mean, so that $R_i^{time}$ is well defined. 
%Finally, $R_i^{space}$ is well defined because $(u_\ell +w_{q+1}).\nabla \theta_{q+1,i}^{(c)}=\diver (\theta_{q+1,i}^{(c)}(u_\ell +w_{q+1}) )$ has zero spatial mean. 

\subsection{Estimate on $\|R_{q+1,i}\|_{L^1}$} \label{sec:estimate_R_q}
%We now estimate in $L^1$ each term in the definition of $R_{q+1,i}$.
Recall that  the estimate on $\|(\rho_{q,i} u_q)_\ell - \rho_{\ell,i} u_\ell\|_{L^1}$ has been already established in \eqref{e:commutatore}. 
By the property \eqref{ts:antidiv} of the antidivergence operator $\mathcal{R}$, Lemma \ref{l:uglylemma} and \eqref{remark:sum is finite} we have
\begin{align*}
\|R_i^{quadr}\|_{L^1}  & \leq \frac C {\lambda_{q+1}} \sum_{n\ge \nop}\sum_{\xi\in \Lambda_i^{[n]}}
\| \chi(\kappa |R_{\ell,i}|- n)
\textstyle{a_{\xi}\left(\frac{R_{\ell,i}}{|R_{\ell,i}|}\right)}\|_{C^2}\|\Theta_{\xi, \mu_{q+1}, n/\kappa} W_{\xi, \mu_{q+1}, n/\kappa}\|_{L^1} \\&\leq C \delta_{q+2} \frac{\lambda_q^{4(1+\alpha)(d+2)+2}}{\lambda_{q+1}} \leq \frac{\delta_{q+2}}{20}.
\end{align*}
To estimate the terms which are linear with respect to the fast variables, we take advantage of the concentration parameter $\mu_{q+1}$. First of all, by Calderon-Zygmund estimates we get
\begin{align*}
\| R_i^{time} \|_{L^1} \le  C \| (\partial_t \theta_{q+1,i}^{(p)})_2+\partial_t \theta_{q+1,i}^{(c)} - m_i \|_{L^1}
\le  \| (\partial_t \theta_{q+1,i}^{(p)})_2 \|_{L^1}+ | \partial_t \theta_{q+1,i}^{(c)} | + | m_i |	.
\end{align*}
Next, notice that
\begin{align}\label{eqn:tpl1}
\| (\partial_t \theta_{q+1,i}^{(p)})_2 \|_{L^1}
& \leq 
C \sum_{n\ge \nop}\sum_{\xi\in \Lambda^{[n]}}\| \partial_t\big[ \chi(\kappa |R_{\ell,i}|- n) \textstyle{a_{\xi}\left( \frac{R_{\ell,i}}{|R_{\ell,i}|} \right)}\big]\|_{C^0}
\|\Theta_{\xi, \mu_{q+1}, n /\kappa} \|_{L^1}
\\& \leq C \delta_{q+2}^{1/p} \lambda_{q}^{3(1+\alpha)(d+2)}\mu_{q+1}^{-d/p'} \leq \frac{\delta_{q+2}}{20 \lambda_{q+1}}.
\end{align}
If  $a_0(N) $ is sufficiently large, from \eqref{eqn:theta_c}, \eqref{eqn:tpl1}, \eqref{eqn:itsolves} and \eqref{eqn:l26} we get

\begin{align*}
& | \partial_t \theta_{q+1,i}^{(c)} |  + | m_i |	
\\ & \le \| g_i \|_{L^\infty}  \left| \sum_{n\ge \nop}\sum_{\xi\in \Lambda^{[n]}} \int \chi(\kappa|R_{\ell,i}| - n ) \textstyle{ a_{\xi} \left( \frac{ R_{\ell,i} }{|R_{\ell,i}|} \right) }	\partial_t \left[ \Theta_{\xi, \mu_{q+1}, n /\kappa}(\lambda_{q+1} t, \lambda_{q+1}x %+ v_{\xi,n} )
) \right]\, dx  \right| + |m_i| 
\\& \qquad +\| g_i \|_{L^\infty} \| (\partial_t \theta_{q+1,i}^{(p)})_2 \|_{L^1}
\\& \le \| g_i \|_{L^\infty} \left| \sum_{n\ge \nop}\sum_{\xi\in \Lambda^{[n]}} \int \chi(\kappa|R_{\ell,i}| - n ) \textstyle{ a_{\xi} \left( \frac{ R_{\ell,i} }{|R_{\ell,i}|} \right) }	\diver \left[ (\Theta_{\xi, \mu_{q+1}, n /\kappa}W_{\xi, \mu_{q+1}, n/\kappa})(\lambda_{q+1} t, \lambda_{q+1}x %+ v_{\xi,n} )
) \right]\, dx  \right| + |m_i|
\\& \qquad+\frac{\delta_{q+2}}{20} 
\\&=
2 \| g_i \|_{L^\infty} \left| \sum_{n\ge \nop}\sum_{\xi\in \Lambda^{[n]}} \int \nabla \left[\chi(\kappa|R_{\ell,i}| - n ) \textstyle{ a_{\xi} \left( \frac{ R_{\ell,i} }{|R_{\ell,i}|} \right) }\right]\cdot  \left[ (\Theta_{\xi, \mu_{q+1}, n /\kappa}W_{\xi, \mu_{q+1}, n/\kappa})(\lambda_{q+1} t, \lambda_{q+1}x %+ v_{\xi,n} )
)-\frac{n}{k}\xi \right]\, dx  \right|
\\& \qquad+\frac{\delta_{q+2}}{20} 
\\& \le \frac{\delta_{q+2}}{20} + 
\frac{2 \| g_i \|_{L^\infty} \sqrt{d}}{\lambda_{q+1}}\sum_{n\ge \nop}\sum_{\xi\in \Lambda^{[n]}} \| \chi(\kappa|R_{\ell,i}| - n ) \textstyle{ a_{\xi} \left( \frac{ R_{\ell,i} }{|R_{\ell,i}|} \right) }\|_{C^2} \| \Theta_{\xi, \mu_{q+1}, n /\kappa}W_{\xi, \mu_{q+1}, n/\kappa}\|_{L^1}
\\& \le \frac{\delta_{q+2}}{20}+ C \| g_i \|_{L^\infty} \lambda_{q+1}^{-1} \delta_{q+2} \lambda_{q}^{4(1+\alpha)(d+2)}
\le \frac{1}{10}\delta_{q+2},
\end{align*}
and using also  \eqref{eqn:theta_c} and \eqref{eqn:ie-2} , we estimate  the error $R_{i}^{space}$%I removed "new" because it can create confusion since we already have a notion of old and new error
\begin{align*}
&\|R_i^{space}\|_{L^1}\leq \|(u_\ell +w_{q+1}) \theta_{q+1,i}^{(c)}\|_{L^1}
\leq (\|u_\ell \|_{L^1}+\|w_{q+1}\|_{L^1}) \|\theta_{q+1,i}^{(c)} \|_{L^\infty} \leq 2  \lambda_q^\alpha (N \lambda_0)^d \lambda_{q+1}^{-1}
\leq\frac{\delta_{q+2}}{20},
\end{align*}
where the last inequality holds up to enlarging $a_0$ depending on $N$. 
%\red{Since $a_0$ does not appear, this is not clear.}
\begin{comment}
Also by by Calderon-Zygmund estimates we get
\begin{align*}
\|R_i^{space}\|_{L^1}&\leq C\|(u_\ell +w_{q+1}).\nabla \theta_{q+1,i}^{(c)}\|_{L^1}\\
&\leq (\|u_\ell \|_{L^1}+\|w_{q+1}\|_{L^1})\|\nabla g_i\|_{L^\infty}\|\theta_{q+1,i}^{(c)} \|_{L^1} \\
&\leq (\|u_\ell \|_{L^{p'}}+\|w_{q+1}\|_{L^{p'}})\|\nabla g_i\|_{C^0}\|\theta_{q+1,i}^{(p)} \|_{L^p}
\leq C.N \|\nabla g_i\|_{C^0} \delta_{q+1},
\end{align*}
where we have used H\"older inequality. 
\end{comment}

We also have that
\begin{align*}
	\| & \theta_{q+1,i}^{(p)} u_\ell + \rho_{\ell,i} w_{q+1}^{(p)}\|_{L^1} \leq \| \theta_{q+1,i}^{(p)}\|_{L^1} \| u_\ell\|_{L^\infty} + \|\rho_{\ell,i}\|_{L^\infty} \|w_{q+1}^{(p)}\|_{L^1}
	\\&\leq  \sum_{n\ge 12} \sum_{\xi \in \Lambda_i^{[n]}} \| \chi(\kappa|R_{\ell,i}| - n ) \textstyle{ a_{\xi} \left( \frac{ R_{\ell,i} }{|R_{\ell,i}|} \right) }\|_{L^\infty} \| \Theta_{\xi, \mu_{q+1},n/\kappa}\|_{L^1} \| u_\ell\|_{L^\infty} \\
	&\qquad+ \|\rho_{\ell,i}\|_{L^\infty} \|\bar \chi(\kappa|R_{\ell,i}| - n)\|_{L^{\infty}}\|W_{\xi,\mu_{q+1},n/\kappa}\|_{L^1}
	\\& \leq C\delta_{q+2}^{1/p} \lambda_{q}^{2(1+\alpha)(d+2)} \mu^{-d/p'}_{q+1} + C\delta_{q+2}^{1/p'} \lambda_{q}^{2(1+\alpha)(d+2)}\mu^{-d/p}_{q+1}
	\leq \frac{\delta_{q+2}}{20}.
\end{align*}
In the last inequality we used  $2\beta b^2\le 1$, the definition of $\gamma$, and  $b(1+1/p)\ge 2(1+\alpha)(d+2)+1$.
%\footnote{EB: nell'ultima disuguaglianza ho usato $2\beta b^2\le 1$ e la definizione di $\gamma$ e $b(1+1/p)\ge 2(1+\alpha)(d+2)+1$.}

Finally, from \eqref{eqn:theta-pert-p} and \eqref{eqn:wcpprime}
\begin{equation}
\begin{split}
\|  ({\rho_{\ell,i}}+ \theta_{q+1,i}^{(p)})w_{q+1} ^{(c)}
\|_{L^1} &\leq
(\|{\rho_{\ell,i}}\|_{C^1}+ \|  \theta_{q+1,i}^{(p)}\|_{L^{p}}) \|w_{q+1}^{(c)}\|_{L^{p'}}
\\&{\leq C \lambda_q^{4(1+\alpha)(d+2)+\alpha} \lambda_{q+1}^{-1}}
\leq \frac{1}{20} \delta_{q+2}
.
\end{split}
\end{equation}

\subsection{Estimates on higher derivatives} \label{sec:higherderivative}

\begin{equation}
\label{eqn:est-rhoC1}
\begin{split}
\| \rho_{q+1,i}\|_{C^1} &\leq \| \rho_{\ell,i}\|_{C^1}+ \| \theta_{q+1,i}\|_{C^1}\\
& \leq \| \rho_{q,i}\|_{C^1}
+ \sum_{n\ge \no} \sum_{\xi\in \Lambda_i^{[n]}} \| \chi(\kappa|R_{\ell,i}| - n ) \textstyle{ a_{\xi} \left( \frac{ R_{\ell,i} }{|R_{\ell,i}|}\right)}\|_{C^1} \|\Theta_{\xi, \mu_{q+1},n/\kappa}(\lambda_{q+1}x) \|_{C^1}
\\&\leq C \lambda_q^\alpha + 
C \lambda_q^{3(1+\alpha)(d+2)} \lambda_{q+1}\mu_{q+1}^{1+d/p}
\leq \lambda_{q+1}^{\alpha}.
\end{split}
\end{equation}
%\footnote{EB: nell'ultima disuguaglianza ho usato }
An entirely similar estimate is valid for $\|\partial_t \rho_{q+1,i}\|_{C^0}$
and the one for $\|u_{\ell}+ w_{q+1,i}^{(p)}\|_{W^{2,r}}$ is analogous. 
Concerning $ w_{q+1,i}^{(c)}$, we use Lemma~\ref{lemma23} and \eqref{remark:sum is finite}
\begin{equation*}
%\label{eqn:est-rhoC1}
\begin{split}
\| w_{q+1,i}^{(c)}\|_{W^{2,r}} &\le 
\sum_{n\ge \nop}\sum_{\xi\in \Lambda_i^{[n]}} \lambda_{q+1}^{}
\| \bar  \chi(\kappa |R_{\ell,i}|-n) \|_{C^4} \|W_{\xi , \mu_{q+1}, n/\kappa} \|_{W^{2,r}}
\\&\leq C \lambda_q^{6 (1+\alpha) (d+2)} \lambda_{q+1}^2\mu_{q+1}^{2+ d(1/p'-1/r)}
\leq  \lambda_{q+1}^{\alpha}.
\end{split}
\end{equation*}
It remains just to estimate
\[
\| \partial_t u_{q+1} \|_{L^1}   \le  \| \partial_t u_\ell \|_{L^1}+\sum_{i=1}^N \|\partial_t w^{(p)}_{q+1,i}\|_{L^1} + \|\partial_t w^{(c)}_{q+1,i}\|_{L^1}.
\]
From \eqref{eqn:Wnorms} and Lemma~\ref{l:uglylemma}
\begin{align*}
\|\partial_t w^{(p)}_{q+1,i} \|_{L^1} & \le \sum_{n\ge \nop}\sum_{\xi\in \Lambda_i^{[n]}} \lambda_{q+1}\| \partial_t W_{\xi, \mu_{q+1}, n/\kappa} \|_{L^1}
+ \| \partial_t \bar{\chi} (\kappa |R_{\ell,i}| - n )  \|_{L^{\infty}}  \|  W_{\xi, \mu_{q+1}, \kappa/n}  \|_{L^1}
\\& C\delta_{q+2}^{2/p'}\lambda_{q}^{(1+2/p')(d(1+\alpha)+1)}\lambda_{q+1}\mu_{q+1}^{1+\gamma(1+d(2/p'-1))}\le \lambda_{q+1}^{2+\gamma(d+1)}\le \lambda_{q+1}^{\alpha}.
\end{align*}
A similar computation is valid for $\| \partial_tw_{q+1,i}^{(c)} \|_{L^1}$.

\section{Proof of main results}

\subsection{Proof of Theorem \ref{p_comp_supp}}
\begin{proof}[Proof of Theorem \ref{p_comp_supp} assuming Proposition \ref{p_inductive}]
Let $\alpha, b,a_0,M>5$, $\beta >0$ be fixed as in Proposition \ref{p_inductive}. Let $a \geq a_0$ be chosen such that 

\begin{align*}
\sum_{q=0}^{+\infty}\delta_{q+1}^{1/p}< \frac{1}{32M},
\\
\sum_{q=0}^{+\infty} \lambda_q^{-1-\alpha } < \frac{1}{32N},
\end{align*}

Let $\{\phi_i\}_{1\leq i\leq N} \subset C^\infty(\T^d)$ be nonnegative functions with mutually disjoint compact supports such that $\int_{\T^d}\phi_i(x)dx=1$, $\{ x \in \T^d: \phi_i (x) >1 \}$ contains a ball of radius $\frac{1}{4N}$, and $\dist(\supp \phi_i, \supp \phi_j) \geq 1/4N$ for $i \neq j$. We also require that $\| \phi_i \|_{C^S} \leq (100 N)^{d +S}$ for any $S \in \N$.
	Let $\chi :[0,1]\rightarrow [0,1]$ be a smooth function such that $\chi\equiv 0$ on $[0,2/5]$, $\chi\equiv 1$ on $[3/5,1]$ with $\chi'$ is compactly supported on $(2/5,3/5)$, and $\| \partial_t \chi \|_{L^\infty} \leq 20$.  Define $\rho_{0,i}(t,x):=(1-\chi(t))+\chi(t)\phi_i(\lambda_0 x)$ and set $u_0\equiv 0$. 
	We also set 
	\begin{equation*}
	R_{0,i}(t):=-\nabla \Delta^{-1}\Big(\partial_t\rho_{0,i}(t)+\diver(\rho_{0,i}(t)u_0(t))\Big)=-\nabla \Delta^{-1}\Big(\partial_t\rho_{0,i}(t)\Big) = - \partial_t \chi  \nabla \Delta^{-1} \left ( \phi_i(\lambda_0 \cdot ) -1 \right ).
	\end{equation*}

We then have $N$ starting triples  $\{(\rho_{0,i}, u_0, R_{0,i})\}_{1\leq i\leq N}$ for our iteration scheme which enjoy \eqref{eqn:CE-R} with $q=0$ 
for any $i=1,\dots,N$.
 Moreover, thanks to Lemma \ref{lemma23}, we have $\| R_{0,i} \|_{L^1} \leq C \lambda_0^{-1}$. Thus \eqref{eqn:ie-1} is satisfied because $2 \beta <1$ 
 %for any $i=1,\dots,N$ 
 (here we have taken $\lambda_0=a_0$ sufficiently large to absorb the constant $C$).
	Next, we have $\| \partial_t \rho_{0,i} \|_{C^0} + \| \rho_{0,i} \|_{C^1} \leq C \lambda_0$. Since $u_0 \equiv 0$ and $\alpha >1$ we conclude that \eqref{eqn:ie-2} is satisfied as well.
	
	Finally we observe that the family of sets $A_i := \{ x \in \T^d : \phi (\lambda_0 x) >1 \}$ for $i=1,\dots,N$ form a $a_0$-open family.
	
	 We can recursively apply Proposition \ref{p_inductive} to obtain a family of sequences $\{(\rho_{q,i}, u_q, R_{q,i})_{q\in\N}\}_{1\leq i \leq N}$ of smooth solutions  to \eqref{eqn:CE-R} and 
%	a  recursive sequence real numbers of $\zeta_q=\zeta_{q-1}-\delta_p^{1/p}$ with $\zeta_0= 1$ 
	such that
	\begin{itemize}
	\item the sequences $\{\rho_{q,i}\}_{q \in \N}$ is Cauchy in $C(L^p)$ and we denote by $\rho_i$ its limit for any $i = 1,\dots,N$,
	\item  the sequence of divergence-free $\{u_{q} \}_{q \in \N}$ is Cauchy in $C(L^{p'} \cap W^{1,r})$ and we denote by $u$ its limit (whose divergence understood in the sense of distribution vanishes).
\end{itemize}	 
Thanks to property  \eqref{eqn:ie-1} we get that $(u, \rho_i)$ solve the continuity equation for any $i=1,\dots,N$. Property (b) and $\inf_{A_i} \rho_{0,i} (t, \cdot) \geq 1$ also yield
\begin{align*}
 \inf_{A_i} \rho_i(t, \cdot) \geq 1 - \sum_{q = 0}^{+ \infty} \delta_{q+1}^{1/p} \geq \frac{1}{2}.
 \end{align*}
This implies that $A_i \subset \supp (\rho_i(t, \cdot ))$ for any $t \in [0,1]$ and any $i=1,\cdots, N$. Thus $\supp(\rho_i (t, \cdot))$ has non-empty interior, and
 \begin{align*}
 \inf_{\T^d \setminus A_i} \rho_i \geq 0.
\end{align*}
So $\rho_{i}$ are nonnegative.

Finally, since $\rho_{0,i} (t, \cdot ) \equiv 1$ for $t \in [0, 2/5]$ and $\sum_{q=0}^{+ \infty}\lambda_{q}^{-1- \alpha} < \frac{1}{\lambda_0^{1+ \alpha}} < \frac{1}{15}$, and by property (c) and (d) of Proposition \ref{p_inductive}, we get that $\rho_{i}(t, \cdot) \equiv 1$ for $t \in [0, 1/3]$ for any $i =1,\dots,N$. Also, by property (d), and since $\sum_{q=0}^{+ \infty}\lambda_{q}^{-1- \alpha} < \frac{1}{\lambda_0 N}$ and $\dist(\supp \rho_{0,i}(1, \cdot ), \supp \rho_{0,j}(1, \cdot )) \geq \frac{1}{\lambda_0 N}$ for $i \neq j$, we must have that $\supp (\rho_{i}) \cap \supp(\rho_j)$ is negligible for $i\neq j$.

\end{proof}

\subsection{Proof of Theorem \ref{t_main}}

\begin{proof}[Proof of Theorem \ref{t_main} assuming Theorem \ref{p_comp_supp}]
	Let $\{\rho_0\}_{1\leq i\leq N}\subset C_t L^p$ be nonnegative densities and $u \in C_t (L^{p'} \cap W^{1,r})$ a divergence-free vector field given by Theorem \ref{p_comp_supp}. Then $(\rho_i, u)$ solves \eqref{eq_continuity} for $i=1,\dots, N$. 
	Thanks to the Ambrosio's superposition principle (see \cite[Theorem 3.2]{A08}), each nonnegative $L^1([0,1] \times \T^d)$ solution is transported by
	% integral curves 
	a generalized flow $\eta^i$ of the vector field $u$.
	More precisely, $\eta^i\in \mathscr{M}_+(AC([0,1];\T^d)\times \T^d)$ is
	%be the probability measure on the space of absolutely continuous curves, 
	concentrated on pairs $(\gamma,x)$ such that $\gamma$ an integral curve of $u$ starting from $x$, and we have
	%for $\eta^i$-a.e.. %on the integral curves of the vector field in the sense of Definition \ref{defn:int-curve} 
	$\rho_i (x,t) \mathscr{L}^d = (e_t)_\sharp \eta^i $ for every $t \in [0,1]$.
	
	Observe that the family of probability measures $\{\eta^i\}_{1\leq i\leq N}$ does not depend on the pointwise representative of $u$. Indeed, given two pointwise representative $v$ and $w$ of $u$ ($u$ and $v$ are two Borel maps such that $v=u=w$ $\mathscr{L}^{d+1}$-a.e.), by Fubini and by the superposition principle, we have  for each integer $1\leq i\leq N$ 
	\begin{equation}\label{eq_not_depend_pointwise}
	\begin{split}
	&\int_{\text{AC}([0,1],\T^d)\times \T^d}\Big(\int_{0}^{1}|v(\gamma(s))-w(\gamma(s))|ds\Big)d\eta^i(\gamma,x)\\
	&=\int_{0}^{1}\Big(\int_{\T^d}|v(y)-w(y)|\rho_i(s,x)d\mathscr{L}^d(x)\Big)ds=0.
	\end{split} 
	\end{equation}
	Thus, $\eta^i$ is concentrated on integral curves of $v$ if and only if $\eta^i$ is concentrated on integral curves of $w$.
	
	By the superposition principle we have %and since $\rho^{n}(1,\cdot)=\phi^{n}(\cdot)$, we have
	\begin{equation}\label{eq_superposition_principle}
	\int_{\T^d} \psi(x) \rho_i(1,x) d\mathscr{L}^d(x)= \int_{\T^d} \int_{\text{AC}([0,1],\T^d)} \psi(\gamma(1)) d\eta^{i}_x(\gamma) d\mathscr{L}^d(x), 
	\end{equation}
	and for every $\psi\in C(\T^d)$. % and for $n=1,\dots, N.$  
	Therefore, for $\mathscr{L}^d$-a.e. $x\in\T^d$ and $\eta^{i}_x$-a.e. $\gamma \in \AC$, we have $\gamma (1)\in A_i$. Since $A_i\cap A_j=\emptyset$ for $i\neq j$, it follows that for $\mathscr{L}^d$-a.e. $x\in \T^d$, the measures $\{\eta^i_x\}_{1\leq i \leq N}$ have mutually disjoint supports. Therefore, for $\mathscr{L}^d$-a.e. $x\in \T^d$ , there are at least $N$ integral curves starting from $x$. 
\end{proof}

\section{Dimension $d=2$}\label{section_dimension2}
The two dimensional case (i.e. for $d=2$) is slightly more technical. We can no longer use Lemma 4.2 of \cite{BDLC20} to translate in space the tubes supporting the building blocks and thereby make these tubes disjoint. 
%Nevertheless, since the buildings blocks which we borrow from \cite{BDLC20} are in fact supported in small time-dependant balls, there is still hope to get the necessary property of disjointness of supports. 
In \cite{BDLC20}, the authors found a way around this issue. They are able make the building blocks in the case $d=2$ disjoint. 
They take  advantage of the presence of a single error to argue that only building blocks with comparable speeds -- that is building blocks for which the speed ratio is of order $\sim$ 10 -- \footnote{this hypothesis is in \cite[Lemma 7.2]{BDLC20} where it is required that $w < 10$} need to have disjoint supports. Indeed in \cite{BDLC20}, the speeds at the inductive step $q \in \N$ of the convex integration scheme are $w_n = \mu_{q+1}^{d/p'} (\frac{n}{\kappa})^{1/p'}$ for $n =1,..., \lam $. The supports of the building blocks are then translated in space suitably, the speeds $\{ w_n \}_{n=1,.., \lam}$ are approximated by $\{ v_n \}_{n=1,.., \lam}$, and at the price of a small error the authors obtain building blocks satisfying
\begin{equation}\label{eq:disjointsupportsdim2}
	W_{\xi, \mu_{q+1}, v_n} \cdot \Theta_{\xi', \mu_{q+1}, v_m}( t , x )=0
	\quad \text{for any $(x,t)\in \T^2\times \R^+$,}
\end{equation}
for any $\xi \neq \xi'$, $|n -m | \leq 1$, $n,m =1,.., \lam.$

However, they deal with only two distinct families $\Lambda_1$ and $\Lambda_2$ of directions for the building blocks in and a single error $R_q$ at each step $q$ of the iteration, whereas  in our setting there are $2N$ distinct families of directions $\{\Lambda_i\}_{1\leq i\leq 2N}$ for the building blocks and $N$ errors $\{R_{q,i}\}_{1\leq i \leq N}$. 
We therefore need that any building block with direction in $\Lambda_i$ has disjoint support with any other building block with direction in $\Lambda_j$ for any $i \neq j$ because in the convex integration scheme we need the key identity \eqref{eqn:formulone_prodotto}  to hold.
% whose time speed difference is of order $\lam$ have disjoint support. 
%We modify the $2$ dimensional building blocks in \cite[Section 7]{BDLC20}, because \textcolor{red}{TO BE EXPLAINED, ESSENTIALLY BECAUSE WE NEED A ``REFINED'' PROPERTY. EXPLAIN THE IDEA}.
%The difficult part will be to choose $v_n$ and $a_{\xi,n}$ so that 
More precisely we will need
\begin{equation}\label{eq:disjointsupportsdim2}
	W_{\xi, \mu_{q+1}, v_n} \cdot \Theta_{\xi', \mu_{q+1}, v_m}( t , x )=0
	\quad \text{for any $(x,t)\in \T^2\times \R^+$,}
\end{equation}
whenever $\xi\neq \xi' \in \Lambda$, $n,m=1,..,\lam $.
%This is possible because the ratio between the time velocities ($w_n= \mu_{q+1}^{d/p'} (\frac{n}{\kappa})^{1/p'}$) 
This identity is achievable because the speed ratios
of the building blocks  are at most of order $\lam$, typically a very small number compared to $\mu_{q+1}$ 
%(recall that $\mu_{q+1}$ is the size of the compact support in space) 
in the iterative proposition (see Section \ref{sec:param}). So we will prove that we can find $ \sim \lam$ balls of radius $\sim \mu_{q+1}^{-1}$ which are moving with speed ratio at most $\sim \lam$ and which don't intersect at any time. We will proceed similarly to \cite[Section 7]{BDLC20}, although our argument differs in some parts for reasons which were outlined above. 
%\begin{lemma}\label{l_disjoint_d2}
%	Fix two different vectors $\xi, \xi'\in \mathbb{S}^{1}\cap \mathbb{Q}^2$ and a number $w=  \frac{A}{N} \leq  \lam $, with $A$ and $N$ positive integers and coprime such that $N \leq  \lam $. Then there exists $C=C(\xi, \xi')$ such that
%		$$\mathscr{L}^1( [0,1] \setminus \{s: \, d_{\T^2}( t \xi, (t w + s)  \xi')\ge \eps  \quad \forall\, t\ge 0  \})< C \eps (\lam)^2. $$
%\end{lemma}
%Towards this aim, we first have the following lemma.
\begin{lemma}\label{l_disjoint_d2}
	%	Fix two different vectors $\xi, \xi'\in \mathbb{S}^{1}\cap \mathbb{Q}^2$ and a number $w=  \frac{A}{N}< \lam $, with $A$ and $N$ positive integers and coprime such that $N < \lam $. 
	Let $\xi, \xi'\in \mathbb{S}^{1}\cap \mathbb{Q}^2$ be two distinct vectors and let $w=  \frac{A}{N}< \lam $ where $A$ and $N$ are positive, coprime integers such that $N < \lam $. 
	Then there exists $C=C(\xi, \xi')$ such that for any $\eps>0$
	$$\mathscr{L}^1( [0,1] \setminus \{s: \, d_{\T^2}( t \xi, (t w + s)  \xi')\ge \eps  \quad \forall\, t\ge 0  \})< C N \eps \lam. $$
\end{lemma}

\begin{proof}
	Let $\eps>0$. Set $T_{int}:=\{ (t,t'): \, \xi t= \xi' t' \, \text{on} \, \T^2 \}$ and observe that $T_{int}\subset \mathbb{Q}^2$ since the matrix with columns $\xi$ and $\xi'$ is invertible  with rational coefficients. Moreover $T_{int}$ is an additive discrete subgroup of $\R^2$, hence it is a free group of rank $k\in \{0,1,2\}$.
    Denoting by $T$ and $T'$ the period of, respectively, $t\to \xi t$ and $t\to \xi' t$ one has that $(T,0),(0,T')\in T_{int}$. This implies that the rank of $T_{int}$ is two, hence we can find two generators $(t_1, t'_1), (t_2, t_2')\in T_{int}$.
 Let us finally introduce $$A:=\{ \xi t\in \T^2: \, (t,s)\in T_{int}\, \text{ for some } s\in \R \}$$ to denote the set of points in $\T^2$ where the supports of the curves $t\to t \xi$ and $t\to t\xi'$ intersect. 
%	\medskip
	\begin{comment}
	 Let $s\in [0, 1]$ be such that there exists $t\ge 0$ satisfying $d_{\T^2}( t\xi, (t+s)\xi')<\eps$.
	There exists $q\in A$ such that
	$d_{\T^2}(t\xi, q)\le \bar c\eps$, where $\bar c=\bar c(\xi,\xi')>1$,
	hence up to modifying $t$ we can assume that $t\xi=:q \in A$ and $d_{\T^2}(q,(t+s) \xi')\le 2\bar c\eps$. Since $t\xi\in A$ there exists $t'$ such that $(t,t')\in T_{int}$ and, exploiting the fact that $(t_1, t'_1), (t_2, t_2')\in T_{int}$ are generators, we can find $k_1, k_2\in \mathbb{Z}$ such that $t= k_1 t_1+k_2 t_2$ and $t'=k_1 t_1' + k_2 t_2'$.  The following identity holds on $\T^2$ 
	\begin{equation*}
		(t+s)\xi' = (k_1 t_1+k_2 t_2) \xi' + s \xi' = (k_1(t_1-t_1')+k_2(t_2-t_2'))\xi' + q + s\xi',
	\end{equation*}
	therefore $d_{\T^2}((k_1(t_1-t_1')+k_2(t_2-t_2')+s)\xi', 0)\le 2\bar c \eps$. This implies $-s\in B_{2\bar c\eps}(k_1(t_1-t_1')+k_2(t_2-t_2')) + \mathbb{Z}T'$.
    Notice that the set $E:=\{  k_1(t_1-t_1')+k_2(t_2-t_2'): \, k_1, k_2\in \mathbb{Z} \}$ is discrete since $t_1-t_1'$, $t_2-t_2'$ are rational numbers. Any two consecutive points in $E$ have a fixed distance $c=c(\xi, \xi')>0$ and $E+\mathbb{Z}T'=E$. In particular
    \begin{equation}
        \mathscr{L}^1 ( [0,1] \setminus  \{s : \, d_{\T^2}( t\xi, (t + s) \xi')\ge \eps \quad \forall\, t\ge 0  \}) \le 
    	\mathscr{L}^{1} \left([-1,1] \cap \bigcup_{r\in E} B_{2\bar c\eps}(r) \right) \le \frac{4 \bar c}{c} \eps.
    \end{equation}
    \end{comment}
 %   \medskip
    
  Let $s\in [0, 1]$ be such that $d_{\T^2}( t\xi, (t+s) w \xi')<\eps$ for some $t\ge 0$.
There exists $q\in A$ such that
	$d_{\T^2}(t\xi, q)\le \bar c\eps$, where $\bar c=\bar c(\xi,\xi')>1$,
	hence up to modifying $t$ we can assume that $t\xi=:q \in A$ and $d_{\T^2}(q,(t w+s) \xi')\le 3 \bar c \eps \lam $. Since $t\xi\in A$ there exists $t'$ such that $(t,t')\in T_{int}$ and, exploiting the fact that $(t_1, t'_1), (t_2, t_2')\in T_{int}$ are generators, we can find $k_1, k_2\in \mathbb{Z}$ such that $t= k_1 t_1+k_2 t_2$ and $t'=k_1 t_1' + k_2 t_2'$.  The following identity holds on $\T^2$ 
    \begin{align*}
    (t w + s) \xi'  & = t' \xi' - t'\xi' + (t w + s) \xi' 
    = q - (k_1 t_1'+k_2 t_2') \xi' + ((k_1 t_1+k_2 t_2) w + s) \xi' 
    \\& =
    q + (k_1(t_1w -t_1')+k_2(t_2 w -t_2') +  s)  \xi'
    \end{align*}
    therefore $d_{\T^2}((k_1(t_1w-t_1')+k_2(t_2w-t_2')+s)\xi', 0)\le 3\bar c \eps \lam $ this implies that $-s\in B_{3\bar c\eps \lam }((k_1(wt_1-t_1') + k_2(wt_2-t_2'))) + \mathbb{Z}T'$.
    
    Notice now that the set $E:=\{  k_1( w t_1-t_1')+k_2( w t_2-t_2'): \, k_1, k_2\in \mathbb{Z} \}$ is discrete, so any two neighbouring points in $E$ are at least a distance $c \ge c'(\xi, \xi')N^{-1}>0$ from each other, and $E + \mathbb{Z} T'= E$. In particular
    \begin{align}
    \mathscr{L}^1 ( [0,1] \setminus & \{s : \, d_{\T^2}( t\xi, (t w + s) \xi')\ge  \eps \quad \forall\, t\ge 0  \})  \le 
    \mathscr{L}^{1} \left([0,1] \cap \bigcup_{r\in E} B_{3\bar c\eps \lam }(r) \right)
    \\ & \le \frac{2}{c'(\xi,\xi') N^{-1}} 3 \bar c \eps \lam 
    \le \frac{6 \bar c }{ c'(\xi,\xi')} \eps (\lam)^2, \qedhere
    \end{align}   
    where in the last we used the inequality $N < \lam.$
\end{proof}

We now need a number theory lemma, it is just a property on real numbers, but we state it for a sequence of real numbers, since we will apply it for a sequence.

\begin{lemma} \label{l_velocities_d2}
Let $\{ \alpha_n \}_{n \in \N} \subset \R^+ $ such that $\alpha_n  \leq \lam $, then there exists $\{ \tilde v_n \}_{n \in \N} \subset \Q$ such that the following holds:
\begin{itemize}
\item $\tilde v_n = a_n + \frac{p_n}{q_n}$, with $p_n, q_n \in \N$,
\item $a_n = \floor*{{\alpha_n}} \leq \lambda_q^\alpha $,
\item $q_n, p_n \leq \lam $,
\item $ 0 \leq \alpha_n - \tilde v_n \leq \frac{2}{\lam }$,
\end{itemize} 
for any $n \in \N.$
\end{lemma}

\begin{proof}
Fix $\alpha_n$, we define $a_n := \floor*{{\alpha_n}} $ and $\overline{\alpha}_n := \alpha_n - a_n \in [0,1).$ We want to approximate $\overline{\alpha}_n$ with dyadic numbers. We define $\ell := \max \{ N \in \N : 2^N \leq \lam  \} $. Since the dyadic intervals are such that 
$$\bigcup_{i=0}^{2^\ell -1} \left [ \frac{i}{2^\ell} , \frac{i+1}{2^\ell} \right ) = [0,1)$$
there exists $i = 0,.., 2^\ell$ such that $\overline{\alpha}_n \in \left [ \frac{i}{2^\ell} , \frac{i+1}{2^\ell} \right )$, defining $p_n =i$ and $q_n = 2^\ell$, we get the thesis.
\end{proof}

\begin{prop} \label{p_disjointd2}
Consider a finite number of disjoint sets $\Lambda_i^j$ for $i=1,..,N$, $j=1,2$, as in Lemma \ref{l_geom} and their union $\Lambda:= \bigcup_{i=1}^N \bigcup_{j=1}^2 \Lambda_i^j \subset \R^2$. Let $\{w_n\}_{n=1,..,\lam }\subset \R$  satisfy
	\begin{equation*}
		w_n = a_n + \frac{p_n}{q_n}
	\end{equation*}
	where $a_n, q_n, p_n$ are positive integers and they are less or equal than $ \lam$.
		Then there exists a constant $c_0:= c_0(\bar C, \Lambda) >0$ with the following property: for every $\xi \in \Lambda$ and $n\in \N$ there exists %$p_{\xi, n}=a_{\xi,n} \xi$ with 
	$a_{\xi,n}\in [0,1]$ such that the family of curves
	\begin{equation}
		x_{\xi,n}(t):= (w_n t+a_{\xi,n}) \xi \qquad \mbox{with $\xi\in \Lambda$,  $n=1,.., \lam $}
	\end{equation}
	satisfies
	\begin{equation}
		d_{\T^2}(x_{\xi,n}(t) , x_{\xi',m}(t)) \ge \frac{c_0}{(\lam)^4}
		\quad \text{for every $t\ge 0$, when  $\xi\neq \xi'$}.
	\end{equation}
\end{prop}
\begin{proof}
We fix $c_0$ such that 
$$ C c_0 |\Lambda| <1.$$
We define the following sets
$$ A_{\xi, \xi', n,m} := \left \{ s\in [0,1] : d_{\T^2} \left ((w_n t + s) \xi, (w_m t + s) \xi' \right ) \geq  \frac{c_0}{(\lam)^4 } \right \}$$
for $\xi, \xi' \in \Lambda$ and $n,m=1,.., \lam.$

We define
$$ A: = \bigcap_{n,m=1}^{\lam } \bigcap_{\xi \neq \xi' \in \Lambda}  A_{\xi, \xi', n,m}$$
and the thesis will follow by proving that $A$ is not empty. We claim that $\mathscr{L}^1 (A) >0.$ Using Lemma  \ref{l_disjoint_d2} we notice that the measure of the complement of the set $A_{\xi, \xi',n,m} $ satisfies $\mathscr{L}^1( A^c_{\xi, \xi',n,m}) \leq  \frac{C c_0 (\lam)^2}{(\lam)^4}= \frac{C c_0 }{(\lam)^2}$.
Then
\begin{align*}
\mathscr{L}^1(A^c) & = \mathscr{L}^1 \left(  \bigcup_{n,m=1}^{\lam } \bigcup_{\xi \neq \xi' \in \Lambda}  A^c_{\xi, \xi', n,m}  \right) 
\\
& \leq \sum_{n,m=1}^{\lam} \sum_{\xi \neq \xi' \in \Lambda} \mathscr{L}^1(A^c_{\xi, \xi', n,m} ) \leq (\lam)^2 |\Lambda|  \frac{C c_0 }{(\lam)^2} \leq C c_0 |\Lambda| <1,
\end{align*}
and so $A$ is not empty. 
\end{proof}

\subsection{Disjointness of the supports}
Set $w_n:= \mu_{q+1}^{d/p'}  v_n^{1/p'}$, where the sequence  $\{\tilde v_n\}_{n= 1,..,\lam }$ is given by Lemma \ref{l_velocities_d2} applied to the sequence $\alpha_n = \left( \frac{n}{k} \right)^{1/p'}$ and $ v_n = \tilde v_n^{p'}$ (notice that the assumption $\alpha_n \leq \lam$ is satisfied thanks to the bound \eqref{remark:sum is finite}). We apply Proposition \ref{p_disjointd2} to $\{ w_n \}_{n= 1,..,\lam }$ (notice that the assumptions are satisfied in view of Lemma \ref{l_velocities_d2}) obtaining the family $\{ a_{\xi, n}:\, \xi\in \Lambda, n=1,..,\lam  \}$. 
Finally, starting from the building blocks introduced in Section~\ref{sec:building}, we define
\begin{equation}
	W_{\xi, \mu_{q+1},  v_n}( t , x ):=\tilde W_{\xi, \mu_{q+1},  v_n}( t, x-a_{\xi,n}\xi),
	\quad
	\Theta_{\xi, \mu_{q+1},  v_n}(t, x):=\tilde \Theta_{\xi, \mu_{q+1},  v_n}(t, x-a_{\xi,n} \xi),
\end{equation}
for any $n= 1,..., \lam $ and $\xi\in \Lambda$. 

We now show that
\begin{equation}\label{eq:disjointsupportsdim2-bis}
	W_{\xi, \mu_{q+1}, v_n} \cdot \Theta_{\xi', \mu_{q+1}, v_m}( t , x )=0
	\quad \text{for any $(x,t)\in \T^2\times \R^+$,}
\end{equation}
for any $\xi\neq \xi' \in \Lambda$, $n,m=1,.., \lam $.

Indeed for any fixed $t\ge 0$ one has the inclusions 
\begin{equation}
     \supp W_{\xi, \mu_{q+1}, v_n}( t , \cdot) \subset B_{2\rho\mu^{-1}_{q+1}}( t w_n \xi + a_{\xi, n}\xi ),
     \quad
     \supp \Theta_{\xi', \mu_{q+1}, v_m}( t , \cdot) \subset B_{\rho\mu^{-1}_{q+1}}( t w_m \xi' + a_{\xi', m} \xi'),
\end{equation}
hence we just need to check that
$B_{2\rho\mu^{-1}_{q+1}}( t w_n \xi + a_{\xi, n}\xi )\cap B_{\rho\mu^{-1}_{q+1}}( t w_m \xi' + a_{\xi', m} \xi')= \emptyset$.
Proposition \ref{p_disjointd2} guarantees 
\[
d_{\T^2}(t w_n \xi + a_{\xi, n}\xi,  t w_m \xi' + a_{\xi', m} \xi')\ge \frac{c_0}{(\lam)^4},
\]
hence the claim is proved provided 
%{\color{red}\begin{equation}\label{zz3}
%	3 \rho \mu_{q+1}^{-1} \le \frac{c_0}{C \kappa \lambda_{q}^{d(1+\alpha)+1}}
%	\quad \text{for $\rho$ small enough.}
%\end{equation}
%}
\begin{equation}\label{zz3}
	\frac 34 \mu_{q+1}^{-1} \le \frac{c_0}{ (\lam)^4 }
	%\quad \text{for $\rho$ small enough.}
	.
\end{equation}
%The proof of
 \eqref{zz3} follows from our choice of $\mu_{q+1} = \lambda_q^{b\gamma}$, because $\gamma>1$ and $b>4 (d(1+\alpha)+2)$.

\subsection{{Proof of the Proposition~\ref{p_inductive} in the case d=2}}
The estimates up to Section~\ref{sec:end} are done in the same way by observing that $\tilde v_n=v_n^{1/p'}$ and $(n/\kappa)^{1/p'} $ are comparable up to a factor $2$.
In Section~\ref{sec:new}, we computed the product  $\theta_{q+1,i}^{(p)} w_{q+1}^{(p)} $ in \eqref{eqn:formulone} with which we were able to compensate the old error $R_{\ell,i}$ (for $i=1,\dots, N$).
Now this product has the form
\begin{equation}
	\theta_{q+1,i}^{(p)}\sum_{k=1}^N w^{(p)}_{q+1,k}= \theta_{q+1,i}^{(p)}w^{(p)}_{q+1,i}=\sum_{n\ge \nop} \sum_{\xi\in \Lambda_i^{[n]}}    \chi(\kappa |R_{\ell,i}|- n)
\textstyle{a_{\xi}\left(\frac{R_{\ell,i}}{|R_{\ell,i}|}\right)} (\Theta_{\xi, \mu_{q+1}, v_n} W_{\xi, \mu_{q+1}, v_n})(\lambda_{q+1} t, \lambda_{q+1}x),
\end{equation}
as a consequence of \eqref{eq:disjointsupportsdim2}, \eqref{remark:sum is finite}, the fact that $\chi\cdot \bar \chi= \chi$ and $\chi ( \kappa | R_{\ell,i} | - n ) \cdot \chi( \kappa| R_{\ell,i}| - m )=0$ when $|n-m|>1$.

Since the average of $\Theta_{\xi, \mu_{q+1}, v_n}W_{\xi, \mu_{q+1}, v_n}$ which appears from the forth line of formula \eqref{eqn:formulone}, in the definition of $R^{quadr}$ and in $m$ is now ${v_n} \xi$ rather than $n/\kappa\xi$, the definition of $\tilde R_{\ell,i}$ should now be replaced by
$$
\tilde R_{\ell,i}:= 
\sum_{n\ge \nop} \chi(\kappa |R_{\ell,i}|- n)	\frac{R_{\ell,i}}{|R_{\ell,i}|} {v_n},
$$
and the obvious modification takes place for the definition of $R^{quadr}$ and $m$.
Observing that $|v_n - \frac{n}{k}| \leq p' \frac{n}{k} |v_n^{1/p'} - (\frac{n}{k})^{1/p'}| \leq p' \lam \frac{\delta_{q+2}}{40 p'} \frac{2}{\lam } = \frac{\delta_{q+2}}{20}$ the estimate \eqref{eqn:tildeRestimate} now works analogously to give $|R_{\ell,i}- \tilde R_{\ell,i}| \leq \frac{16%\no
}{20%\no
} \delta_{q+2}$.
The rest of the estimates work as in Sections~\ref{sec:new}, \ref{sec:estimate_R_q} and~\ref{sec:higherderivative}.

\vspace{0.5cm}

\textbf{ Acknowledgements}. 
MS has been supported by the SNSF Grant 182565.
The authors wish to thank Maria Colombo for bringing the problem of non-uniqueness of integral curves to their attention and for useful suggestions.

\bibliographystyle{alpha}
\bibliography{bibliografia}

\end{document}